\documentclass[11pt,english]{amsart}
\usepackage[utf8]{inputenc}
\usepackage{ae,aecompl}

\usepackage[T1]{fontenc}
\usepackage{geometry}
\usepackage{color}
\usepackage{babel}
\usepackage{enumitem}
\usepackage{mathrsfs}
\usepackage{amstext}
\usepackage{amsthm}
\usepackage{amssymb}
\usepackage{comment}

\usepackage{mathtools} \mathtoolsset{showonlyrefs,showmanualtags}
\usepackage[
    colorlinks=true,       % false: boxed links; true: colored links
    linkcolor=magenta,          % color of internal links
    citecolor=magenta,        % color of links to bibliography
    filecolor=magenta,      % color of file links
    urlcolor=magenta,           % color of external links
    pagebackref=true]{hyperref}

\usepackage{stmaryrd,mathrsfs}

\makeatletter
\numberwithin{equation}{section}
\numberwithin{figure}{section}
\theoremstyle{definition}
\ifx\thechapter\undefined
  \newtheorem{defn}{\protect\definitionname}
\else
  
\fi
\theoremstyle{plain}
\ifx\thechapter\undefined
  \newtheorem{thm}{\protect\theoremname}
\else
  \newtheorem{thm}{\protect\theoremname}[chapter]
\fi
\theoremstyle{remark}
\ifx\thechapter\undefined
  \newtheorem{rem}{\protect\remarkname}
\else
  \newtheorem{rem}{\protect\remarkname}[chapter]
\fi
\theoremstyle{plain}
\ifx\thechapter\undefined
  \newtheorem{prop}{\protect\propositionname}
\else
  
\fi
\theoremstyle{plain}
\ifx\thechapter\undefined
  \newtheorem{lem}{\protect\lemmaname}
\else
  \newtheorem{lem}{\protect\lemmaname}[chapter]
\fi

\usepackage{etex}
\usepackage{amsfonts}
\usepackage{amsthm}
\usepackage{array}
\usepackage{amsmath}
\usepackage{float}

\DeclareMathOperator {\supp}{supp}
\date{}
\makeatletter
\makeatletter
\def\clearpage{%
  \ifvmode
    \ifnum \@dbltopnum =\m@ne
      \ifdim \pagetotal <\topskip
        \hbox{}
      \fi
    \fi
  \fi
  \newpage
  \thispagestyle{empty}
  \write\m@ne{}
  \vbox{}
  \penalty -\@Mi
}
\makeatother

\newcommand{\Rolosaysnopage}[1]{{ }}

\usepackage{pgf,tikz}
\usetikzlibrary{arrows}

\def\Xint#1{\mathchoice
{\XXint\displaystyle\textstyle{#1}}%
{\XXint\textstyle\scriptstyle{#1}}%
{\XXint\scriptstyle\scriptscriptstyle{#1}}%
{\XXint\scriptscriptstyle\scriptscriptstyle{#1}}%
\!\int}
\def\XXint#1#2#3{{\setbox0=\hbox{$#1{#2#3}{\int}$}
\vcenter{\hbox{$#2#3$}}\kern-.5\wd0}}

\def\dashint{\Xint-}

\makeatletter
\newsavebox{\@brx}
\newcommand{\llangle}[1][]{\savebox{\@brx}{\(\m@th{#1\langle}\)}%
  \mathopen{\copy\@brx\kern-0.5\wd\@brx\usebox{\@brx}}}
\newcommand{\rrangle}[1][]{\savebox{\@brx}{\(\m@th{#1\rangle}\)}%
  \mathclose{\copy\@brx\kern-0.5\wd\@brx\usebox{\@brx}}}
\makeatother

\makeatother

\providecommand{\definitionname}{Definition}
\providecommand{\lemmaname}{Lemma}
\providecommand{\propositionname}{Proposition}
\providecommand{\remarkname}{Remark}
\providecommand{\theoremname}{Theorem}

\DeclareMathOperator*{\esssinf}{ess\ inf}
\allowdisplaybreaks
\begin{document}

\title{Quantitative matrix weighted estimates for certain singular integral operators}

\author{Pamela A. Muller}

\address{Universidad Nacional del Sur - Instituto de Matemática de Bahía Blanca,
Bahía Blanca, Argentina.}

\email{pamela.muller@uns.edu.ar}

\author{Israel P. Rivera-Ríos }

\address{Universidad Nacional del Sur - Instituto de Matemática de Bahía Blanca,
Bahía Blanca, Argentina.}

\email{israel.rivera@uns.edu.ar}
\begin{abstract}
In this paper quantitative weighted matrix estimates for vector valued extensions of $L^{r'}$-H\"ormander operators and rough singular integrals are studied. Strong type $(p,p)$ estimates, endpoint estimates, and some new results on Coifman-Fefferman estimates assuming $A_\infty$ and $C_p$ condition counterparts are provided. To prove the aforementioned estimates we rely upon some suitable convex body domination results that we settle as well in this paper.
\end{abstract}

\maketitle

\section{Introduction}
We recall that given $p>1$, a non negative locally integrable function
$w$ is an $A_{p}$ weight if 
\[
[w]_{A_{p}}=\sup_{Q}\frac{1}{|Q|}\int_{Q}w\left(\frac{1}{|Q|}\int_{Q}w^{-\frac{1}{p-1}}\right)^{p-1}<\infty
\]
and an $A_{1}$ weight if 
\[
[w]_{A_{1}}=\sup_{Q}\esssinf_{y\in Q}\frac{1}{|Q|}\int_{Q}w(x)w^{-1}(y)dx<\infty.
\]
These classes were introduced by Muckenhoupt to characterize
the weighted $L^{p}(w)$ boundedness of the Hardy-Littlewood maximal
function. Soon after Muckenhoupt's seminal work, a number of authors
such as Muckenhoupt himself, Wheeden, Hunt, Coifman, Fefferman, among
others devoted some works to study the relationship between singular
integrals and these classes of weights. 

The theory of weights has been a fruitful area of research since then,
with the study of a number of operators and settings too wide for us to be
able to sum it up in a few lines. For a long time the results in the literature of the area
were qualitative, in the sense that the dependence of the inequalities
on the weight or weights involved was not quantified in any sense. However, in the last decade it became
a trending topic in the area the study of the so called quantitative
estimates, namely estimates in which the dependence constants $[w]_{A_{p}}$ and $[w]_{A_{1}}$ was made explicit
and in which the best dependence in some sense was pursued. One of
the fundamental problems in this field that has motivated a large amount
of research, was solved by Hyt\"onen in \cite{H}, in which the so called $A_{2}$ conjecture
was settled. That conjecture, now theorem, says that for every Calderón-Zygmund operator $T$, 
\[
\|Tf\|_{L^{2}(w)}\leq c_{d,T}[w]_{A_{2}}\|f\|_{L^{2}(w)}.
\]
The efforts to understand better this question led to the development
of the sparse domination theory, that started with the seminal
work of Lerner \cite{L}, which has proved to be a powerful tool to
study quantitative estimates. 

Also the study of quantitative estimates lead to try to achive a further
understanding of those questions deriving in estimates in terms of
the $A_{p}$ constants and the $A_{\infty}$ constant. We recall that
the $w\in A_{\infty}=\bigcup_{p\geq1}A_{p}$ if and only 
\[
[w]_{A_{\infty}}=\sup_{Q}\frac{1}{w(Q)}\int_{Q}M(\chi_{Q}w)<\infty.
\]
This constant was introduced by Fujii \cite{F} and rediscovered by
\cite{W} and was shown to be an interesting object of study for quantitative
estimates for first in \cite{HPAinfty}, due to the fact that as it was shown in that work, if $w\in A_p$ then $[w]_{A_\infty}\leq c_d[w]_{A_\infty}$. Since that work also a number of papers have been devoted to study as
well properties of spaces of functions and boudedness of operators
in terms of $[w]_{A_{\infty}}$. 

Vector valued extensions are one of the possible generalizations of
the classical scalar theory. That field of research has received the attention
of a number of authors during the last years. See for instance \cite{DKPS} for a very recent extension of the theory to the biparametric setting. 

Let $W:\mathbb{R}^{d}\rightarrow\mathbb{R}^{n\times n}$
be a matrix weight, namely a matrix function such that $W(x)$ is self-adjoint and
positive definite a.e. 

Given $f:\mathbb{R}^{d}\rightarrow\mathbb{R}^{n}$ and $1<p<\infty$,
we define
\[
\|\vec{f}\|_{L^{p}(W)}=\left(\int_{\mathbb{R}^{d}}\left|W^{\frac{1}{p}}(x)\vec{f}(x)\right|^{p}dx\right)^{\frac{1}{p}}.
\]
Let $1<p<\infty$. We say that a matrix weight $W$ is an $A_{p}$
weight if 
\begin{equation}\label{eq:Ap}
[W]_{A_{p}}=\sup_{Q}\frac{1}{|Q|}\int_{Q}\left(\frac{1}{|Q|}\int\left|W^{\frac{1}{p}}(x)W^{-\frac{1}{p}}(y)\right|_{op}^{p'}dy\right)^{\frac{p}{p'}}dx<\infty
\end{equation}
and if $p=1$ that $W\in A_1$ if
\begin{equation}\label{eq:A1}
[W]_{A_1}=\sup_Q{\displaystyle\esssinf_{y\in Q}}\frac{1}{|Q|}\int_Q|W(x)W^{-1}(y)|_{op}dx.
\end{equation}

Above and throughout the remainder of the paper $|A|_{op}$ stands for the norm as an operator of the matrix $A$, namely 
\[|A|_{op}=\sup_{e\in\mathbb{R}^n\setminus\{0\}}\frac{|A\vec{e}|}{|\vec{e}|}\]

Treil and Volberg \cite{TV} were the first in studying these weights
and their connection with singular integrals. Later on Goldberg \cite{G}
further explored that connection and provided results for certain
maximal functions, and also Nazarov and Treil \cite{NT} and Volberg
\cite{V} further studied the boundedness of Calderón-Zygmund operators.
At this point we would like to note that the definition of the $A_{p}$
class that we have just presented here seems to have appeared for first in \cite{R} and
is equivalent to the definitions in the aforementioned works. The definition of the matrix $A_1$ condition is due to Roudenko and Frazier \cite{FR}.

Contrary to what happens in the scalar case, it is not known whether for $p=2$ the dependence on the matrix $A_2$ constant of Calder\'on-Zygmund operators is linear or not. The current record is due to Nazarov,
Petermichl, Treil and Volberg \cite{NPTV}, who showed that
\[
\|T\vec{f}\|_{L^{2}(W)}\leq c_{n,T}[W]_{A_2}^{\frac{3}{2}}\|\vec{f}\|_{L^{2}(W)}.
\]
That estimate was generalized for $p\not=2$ in \cite{CUIM}.
The aforementioned results rely upon a suitable adaption of the sparse domination,
the so called convex body domination. In the case of the maximal function, the dependence of the scalar case was retrieved in \cite{IKP}, and the only case of singular operator up until now in which the sharp dependence has been settled is the square function \cite{HPV}. 

In the case of estimates in terms of the $A_{1}$ constant, the best dependences have been retrieved for several operators in \cite{IPRR}. 

\section{Main results}

In this paper our purpose is to provide strong type and and endpoint
weak type quantitative estimates for vector valued extensions of rough
singular integrals and $L^{r'}$-Hörmander operators. Some results
had already been obtained for maximal rough singular integrals in
\cite{DiPHL}. In the case of $L^{r'}$-Hörmander operators we are
not aware of any result in this direction. 

We will also provide Coifman-Fefferman estimates going beyond the
$A_{\infty}$ condition and providing a counterpart of the $C_{p}$
condition in this setting.

A key ingredient in our proofs is convex body domination for all the aforementioned operators. We will as well provide convex body domination results for both $L^{r'}$-Hörmander operators and rough singular integrals, relying upon ideas of \cite{NPTV,DiPHL,Le,Li}.

To present our results we need a few more definitions. 

We say that $T$ is an $L^{r}$-Hörmander singular operator if $T$
is bounded on $L^{2}$ and it admits the following representation
\begin{equation}
Tf(x)=\int_{\mathbb{R}^{n}}K(x,y)f(y)dy\label{eq:Rep}
\end{equation}
provided that $f\in\mathcal{C}_{c}^{\infty}$ and $x\not\in\supp f$
where $K:\mathbb{R}^{d}\times\mathbb{R}^{d}\setminus\left\{ (x,x)\,:\,x\in\mathbb{R}^{d}\right\} \rightarrow\mathbb{R}$
is a locally integrable kernel satisfying the $L^{r}$-Hörmander condition,
namely 
\begin{align*}
H_{r,1} & =\sup_{Q}\sup_{x,z\in\frac{1}{2}Q}\sum_{k=1}^{\infty}\left(2^{k}l(Q)\right)^{d}\left\Vert \left(K(x,\cdot)-K(z,\cdot)\right)\chi_{2^{k}Q\setminus2^{k-1}Q}\right\Vert _{L^{r},2^{k}Q}<\infty.\\
H_{r,2} & =\sup_{Q}\sup_{x,z\in\frac{1}{2}Q}\sum_{k=1}^{\infty}\left(2^{k}l(Q)\right)^{d}\left\Vert \left(K(\cdot,x)-K(\cdot,z)\right)\chi_{2^{k}Q\setminus2^{k-1}Q}\right\Vert _{L^{r},2^{k}Q}<\infty.
\end{align*}

Given $\Omega\in L^{\infty}(\mathbb{S}^{d-1})$ with $\int_{\mathbb{S}^{d-1}}\Omega=0$,
we define the rough singular interal $T_{\Omega}$ as 
\[
T_{\Omega}(f)=\lim_{\varepsilon\rightarrow0}\int_{|x-y|>\varepsilon}\frac{\Omega\left(\frac{x-y}{|x-y|}\right)}{|x-y|^{d}}f(y)dy
\]

Let $p\geq1$. we say that a weight $W\in A_{\infty,p}^{sc}$ if
\begin{equation}\label{eq:AinftyMat}
[W]_{A_{\infty,p}^{sc}}=\sup_{e\in\mathbb{\mathbb{R}}^{n}}[|W\vec{e}|^{p}]_{A_{\infty}}<\infty.
\end{equation}

We recall that given a linear operator $T$ and an orthonormal basis ${e_j}$ of $\mathbb{R}^n$ we define the vector valued extension of $T$ by
\[T(\vec{g})(x)=\sum_jT(\langle\vec{g}, e_j\rangle)(x)e_j.\]
It is worth noting that this expression is independent of the basis chosen.

In the following subsections we gather the statements of our main
results and some further comments.

\subsection{$A_p$ estimates}
The results in this subsection provide counterparts for rough singular integrals and $L^{r'}$-H\"ormander of the estimates obtained in \cite{NPTV} and \cite{CUIM} for Calder\'on-Zygmund operators. Since we push forward techniques in those papers also, as one may expect, the estimates obtained do not match with the ones obtained in the scalar case.

\begin{thm}
\label{Thm:RoughAp}Let $\Omega\in L^\infty(\mathbb{S}^{d-1})$ with $\int_{\mathbb{S}^{d-1}}\Omega=0$ and $1<p<\infty$. Then if $W\in A_p$ we have that
\[\|T_\Omega\vec{f}\|_{L^p(W)}\lesssim \|\Omega\|_{L^\infty(\mathbb{S}^{d-1})}[W]_{A_p}^{\frac{1}{p}}[W^{-\frac{p'}{p}}]_{A^{sc}_{\infty,p'}}^\frac{1}{p}[W]_{A^{sc}_{\infty,p}}^{\frac{1}{p'}} \min\left\{[W]_{A^{sc}_{\infty,p}},[W^{-\frac{p'}{p}}]_{A^{sc}_{\infty,p'}}\right\}\|\vec{f}\|_{L^p(W)}\]
\end{thm}
Note that this estimate was sketched for the case $p=2$ in \cite[Remark 6.6]{DiPHL}. Here we provide a full proof and extend the result to every $p>1$.

\begin{thm}
\label{Thm:HormAp}Let $1<r<p<\infty$ and let $T$ be a $L^{r'}$-H\"ormander operator.  Then if $W\in A_{p/r}$ we have that
\[\|T\vec{f}\|_{L^p(W)}\lesssim [W]_{A_{\frac{p}{r}}}^{\frac{1}{p}} [W^{-\frac{r}{p}(\frac{p}{r})'}]_{A^{sc}_{\infty,(\frac{p}{r})'}}^{\frac{1}{p}} [W]_{A^{sc}_{\infty,\frac{p}{r}}}^{\frac{1}{p'}}\|\vec{f}\|_{L^p(W)}\]
\end{thm}

\subsection{$A_1$ and $A_q$ estimates}

Our results here are the counterparts for rough singular integrals and $L^{r'}$-H\"ormander operators of the results obtained in \cite{IPRR}. In these cases we recover as well the optimal estimates known in the scalar case (see \cite{LPRRR}). We remit the interested reader to \cite{IPRR}, to read about further references about the motivation to study this kind of estimates.

\begin{thm} \label{Thm:A1}Let $W\in A_1$. We have that 
\begin{itemize} 
\item If $\Omega\in L^\infty(\mathbb{S}^{d-1})$ with $\int_{\mathbb{S}^{d-1}}\Omega=0$ and $1<p<\infty$ then 
\begin{equation}\label{eq:RoughA1}\|T_\Omega \vec{f}\|_{L^{p}(W)}\leq c_{n,p,d}\|\Omega\|_{L^\infty(\mathbb{S}^{d-1})} [W]_{A_{1}}^{\frac{1}{p}}[W]_{A^{sc}_{\infty,1}}^{\frac{1}{p'}}\|\vec{f}\|_{L^{p}(W)}\end{equation}
\item If $T$ is a $L^{r'}$-H\"ormander operator and $p>r$, then
\begin{equation}\label{eq:HormA1}\|T \vec{f}\|_{L^{p}(W)}\leq c_{n,T,p,d}\left(\frac{p}{r}\right)' [W]_{A_{1}}^{\frac{1}{p}}[W]_{A^{sc}_{\infty,1}}^{\frac{1}{p'}}\|\vec{f}\|_{L^{p}(W)}\end{equation}
\end{itemize}
\end{thm}

\begin{thm} \label{Thm:Aq} Let $1<q<p$ and $W\in A_q$. We have that 
\begin{itemize} 
\item If $\Omega\in L^\infty(\mathbb{S}^{d-1})$ with $\int_{\mathbb{S}^{d-1}}\Omega=0$ then 
\begin{equation}\label{eq:RoughAq}\|T_\Omega \vec{f}\|_{L^{p}(W)}\leq c_{n,p,d}\|\Omega\|_{L^\infty(\mathbb{S}^{d-1})} [W]_{A_{q}}^{\frac{1}{p}}[W]_{A^{sc}_{\infty,q}}^{\frac{1}{p'}}\|\vec{f}\|_{L^{p}(W)}\end{equation}
\item If $T$ is a $L^{r'}$-H\"ormander operator and $\frac{p}{q}>r$, then
\begin{equation}\label{eq:HormAq}\|T \vec{f}\|_{L^{p}(W)}\leq c_{n,T,p,d}\left(\frac{p}{rq}\right)' [W]_{A_{q}}^{\frac{1}{p}}[W]_{A^{sc}_{\infty,q}}^{\frac{1}{p'}}\|\vec{f}\|_{L^{p}(W)}\end{equation}
\end{itemize}
\end{thm}

\subsection{Coifman-Fefferman estimates}
We recall that the classical Coifman-Fefferman inequality asserts that if $T$ is a Calder\'on-Zygmund operator, $0<p<\infty$ and $w\in A_\infty$
\begin{equation}\label{eq:CF}
\|Tf\|_{L^p(w)} \lesssim c_w\|Mf\|_{L^p(w)}.
\end{equation}
Note that a quantitative version with $c_w=[w]_{A_\infty}$ was obtained in \cite{OCPR}. In the case of rough singular integrals the corresponding quantitative counterpart was settled in \cite{LPRRR} and for $L^{r'}$-H\"ormander operators, for instance in \cite{IFRR}. Our vector valued counterpart is the following result.
\begin{thm}
Let $p>1$. Then
\begin{enumerate}
\item If $W\in A_{\infty,p}^{sc}$ and $T$ is a Calderón-Zygmund operator
\[
\||W^{\frac{1}{p}}T(W^{-\frac{1}{p}}\vec{f})|\|_{L^p(\mathbb{R}^d)}\lesssim[W]_{A_{\infty,p}^{sc}}^{\frac{1}{p}}\left\Vert \sup_{Q}\frac{1}{|Q|}\int_{Q}|\mathcal{W}{}_{p,Q}W^{-\frac{1}{p}}\vec{f}|\right\Vert _{L^{p}(\mathbb{R}^{d})}.
\]
\item If $W\in A_{\infty,p}^{sc}$ and $\Omega\in L^{\infty}(\mathbb{S}^{d-1})$
with $\int_{\mathbb{S}^{d-1}}\Omega=0$ then
\[
\||W^{\frac{1}{p}}T_\Omega(W^{-\frac{1}{p}}\vec{f})|\|_{L^p(\mathbb{R}^d)}\lesssim[W]_{A_{\infty,p}^{sc}}^{1+\frac{1}{p}}\left\Vert \sup_{Q}\frac{1}{|Q|}\int_{Q}|\mathcal{W}{}_{p,Q}W^{-\frac{1}{p}}\vec{f}|\right\Vert _{L^{p}(\mathbb{R}^{d})}.
\]
\item If $W\in A_{\infty,\frac{p}{r}}^{sc}$ and $T$ is a $L^{r'}$-Hörmander
operator and $p>r$ then
\[
\||W^{\frac{1}{p}}T(W^{-\frac{1}{p}}\vec{f})|\|_{L^p(\mathbb{R}^d)}\lesssim[W]_{A_{\infty,\frac{p}{r}}^{sc}}^{\frac{1}{p}}\left\Vert \sup_{Q}\left(\frac{1}{|Q|}\int_{Q}|\mathcal{W}{}_{\frac{p}{r},Q}^{\frac{1}{r}}W^{-\frac{1}{p}}\vec{f}|^{r}\right)^{\frac{1}{r}}\right\Vert _{L^{p}(\mathbb{R}^{d})}.
\]
\end{enumerate}
\end{thm}
\begin{rem}
At this point we would like to note that even though the dependence
may look better than in the scalar case, the maximal operator in the
right hand side is a weighted maximal operator, in contrast with the situation in the classical setting. Hence, in some sense, the ``missing'' piece of constant is in disguise ``inside'' the maximal operator. In any case, in this setting, due to the non-linearity of the maximal function, that leads to study weighted versions of it, those inequalities seem a suitable candidate.
\end{rem}
\begin{rem}
The estimate in the case of $L^{r'}$-H\"ormander operators in terms of an $L^r$ maximal function seems the best one may expect in view of the fact that this is the same that happens with scalar $L^{r'}$-H\"ormander operators. We remit the reader to \cite{MPTG} for more details.
\end{rem}

In the scalar case, Muckenhoupt \cite{M} showed that $A_{\infty}$ is not necessary
for the Coifman-Fefferman estimate to hold. He showed that if $p>1$ and
\eqref{eq:CF} holds for the Hilbert transform and a certain weight $w$ then there exists $c,\delta>0$ such that for every cube $Q$ and every measurable subset $E\subset Q$ we have that
\[
w(E)\leq c\left(\frac{|E|}{|Q|}\right)^\delta\int_{\mathbb{R}^n}M(\chi_Q)^pw.
\]
In the 80s Sawyer \cite{S} extended that result to higher dimensions and also showed that for $p>1$ the $C_{p+\varepsilon}$ was sufficient for \eqref{eq:CF} to hold. It is still unknown whether $C_p$ is sufficient for \eqref{eq:CF} to hold.

In the last years several advances have been made, for instance the extension in \cite{CLPRR} to the full range $0<p<\infty$ and other operators relying upon \cite{Y,LeCF1} and sparse domination techniques, the characterization of the good weights for the weak type counterpart of \eqref{eq:CF} in \cite{LeCF2}, and the quantitative results introduced in \cite{C} and further explored in \cite{CLRT}.

In \cite[Theorem 2.5]{C} the following reverse H\"older type inequality was settled for $C_p$ weights. There exists $r>1$ such that
\[\left(\frac{1}{|Q|}\int_Qw^r\right)^{\frac{1}{r}}\lesssim\frac{1}{|Q|}\int_{\mathbb{R}^d}M(\chi_Q)^pw
\]
In the matrix setting the right hand side of that expression seems difficult to ``reproduce''. We recall that the matrix $A_\infty$ conditions are introduced via scalar $A_\infty$ and frequently arise in reverse H\"older inequalities. Taking that into account a definition in terms of a certain reverse H\"older inequality seems reasonable. Those ideas motivate the following definition. Given $1\leq p<q$ we say that $W\in C_{p,q}$ if there exists $\gamma>1$
such that
\[
\langle|\mathcal{W}{}_{p,Q}^{-1}W^{\frac{1}{p}}|_{op}^{\gamma p}\rangle_{Q}^{\frac{1}{\gamma}}\lesssim\frac{1}{|Q|}\int_{\mathbb{R}^{n}}M(\chi_{Q})^{q}.
\]
We remit the reader to Section \ref{sec:defsLemmata} for the precise definition of $\mathcal{W}_{p,Q}$.

\begin{thm}
Given $1< p<q$ 
\begin{enumerate}
\item If $W\in C_{p,q}$ and $T$ is a Calderón-Zygmund operator
\[
\||W^{\frac{1}{p}}T(W^{-\frac{1}{p}}\vec{f})|\|_{L^p(\mathbb{R}^d)}\lesssim\left\Vert \sup_{Q}\frac{1}{|Q|}\int_{Q}|\mathcal{W}{}_{p,Q}W^{-\frac{1}{p}}\vec{f}|\right\Vert _{L^{p}(\mathbb{R}^{d})}.
\]
\item If $W\in C_{p,q}$ and $\Omega\in L^{\infty}(\mathbb{S}^{d-1})$ with
$\int_{\mathbb{S}^{d-1}}\Omega=0$ then
\[
\||W^{\frac{1}{p}}T_\Omega(W^{-\frac{1}{p}}\vec{f})|\|_{L^p(\mathbb{R}^d)}\lesssim\left\Vert \sup_{Q}\frac{1}{|Q|}\int_{Q}|\mathcal{W}{}_{p,Q}W^{-\frac{1}{p}}\vec{f}|\right\Vert _{L^{p}(\mathbb{R}^{d})}.
\]
\item If $r>1$ $W\in C_{\frac{p}{r},q}$ and $T$ is a $L^{r'}$-Hörmander
operator and $p>r$ then
\[
\||W^{\frac{1}{p}}T(W^{-\frac{1}{p}}\vec{f})|\|_{L^p(\mathbb{R}^d)}\lesssim\left\Vert \sup_{Q}\left(\frac{1}{|Q|}\int_{Q}|\mathcal{W}{}_{\frac{p}{r},Q}^{\frac{1}{r}}W^{-\frac{1}{p}}\vec{f}|^{r}\right)^{\frac{1}{r}}\right\Vert _{L^{p}(\mathbb{R}^{d})}.
\]
\end{enumerate}
\end{thm}

\subsection{Endpoint estimates}\label{sec:endpoint}
The study of endpoint estimates for vector valued extensions was initiated in \cite{CUIMPRR}. It is not clear how to make sense of a matrix weight in the role of a density. Note that to study strong type weighted inequalities such as 
\[\|T(\vec{f})\|_{L^p(W)}\lesssim c_W \|\vec{f}\|_{L^p(W)}\]
we usually rewrite the problem as
\[\||W^{\frac{1}{p}}T(W^{-\frac{1}{p}}\vec{f})|\|_{L^p}\lesssim c_W \|\vec{f}\|_{L^p}.\]
Furthermore in the case of the maximal function, since it is not linear, a usual choice is to consider a weighted version of such operator and to study its unweighted estimates. Hence in the case of  endpoint estimates it seems reasonable to study unweighted estimates of ``weighted'' operators, namely, to study
\[\||W^{\frac{1}{p}}T(W^{-\frac{1}{p}}\vec{f})|\|_{L^{1,\infty}}\lesssim c_W \|\vec{f}\|_{L^1}.\]
Quantitative estimates in this direction still seem to be far from from the optimal estimates known in the scalar case. The current record for Calder\'on-Zygmund operators in terms of the $A_1$ constant is $[W]^2_{A_1}$ (see  \cite{CUIMPRR}) while in the scalar setting the sharp bound has already been achieved and is $[w]_{A_1}\log(e+[w]_{A_1})$ (see \cite{LOP,LOP1,LNO}).

Before presenting our results for rough singular integrals and $L^{r'}$-H\"ormander operators, we would like to note that in the scalar setting this kind of estimates are the so called mixed weak type inequalities. First results in this direction are due to Muckenhoupt and Wheeden \cite{MW} and Sawyer \cite{S2} and a number of contributions have been made in the last years. We remit, for instance, to \cite{Be} and to \cite{LiOP} and the references therein for some of those contributions.

Now we present our results.

\begin{thm}\label{eq:RoughEndpoint}
If $\Omega\in L^{\infty}(\mathbb{S}^{d-1})$ with $\int_{\mathbb{S}^{d-1}}\Omega=0$,
then 
\[
\||W(x)T_{\Omega}(W^{-1}f)(x)|\|_{L^{1,\infty}}\lesssim\|\Omega\|_{L^{\infty}(\mathbb{S}^{d-1})}[W]_{A_{1}}[W]_{A_{\infty,1}^{sc}}\max\left\{ \log\left([W]_{A_{1}}+e\right),[W]_{A_{\infty,1}^{sc}}\right\} \|f\|_{L^{1}}.
\]
\end{thm}

\begin{thm}\label{eq:HormanderEndpoint}
Let $W\in A_1$ and let $T$ be a $L^{r'}$-H\"ormander operator. Then
\[\left\Vert |W^{\frac{1}{r}}(x)T(W^{-\frac{1}{r}}\vec{f})(x)|\right\Vert _{L^{r,\infty}(\mathbb{R}^{d})}\lesssim[W]_{A_{1}}^{\frac{1}{r}}[W]_{A_{\infty,1}^{sc}}\||\vec{f}|\|_{L^{r}}\]
\end{thm}

The remainder of the paper is organized as follows. Section \ref{sec:Convex} is devoted to the presentation of convex body domination results for rough singular integrals and $L^{r'}$-H\"ormander operators. In Section \ref{sec:defsLemmata} we provide some further definitions and lemmata. The remainder of the sections are devoted to settle the main results. 

\section{Convex body domination results}\label{sec:Convex}

We recall that a family of cubes $\mathcal{S}$ is $\eta$-sparse for some $\eta\in(0,1)$ if for each $Q\in\mathcal{S}$ there exists $E_Q\subset Q$ such that the sets $E_Q$ are pairwise disjoint and $\eta|Q|\leq |E_Q|$. As it was shown in \cite{LN} a family $\mathcal{S}$ is $\eta$-sparse if and only if $\mathcal{S}$ is $\frac{1}{\mathcal{\eta}}$-Carleson that is, if for each $Q\in\mathcal{S}$
\[\sum_{P\subseteq Q,P\in\mathcal{S}}|P|\leq\frac{1}{\eta}|Q|.\]

Convex body domination was introduced by Nazarov, Petermichl, Treil and Volberg who settled in \cite{NPTV} a ``pointwise'' domination result for Calder\'on-Zygmund operators (see \cite{CDiPOu} for a ``bilinear'' version of that result). Those techniques where also explored for commutators in \cite{CUIMPRR,IPRR,IPT} and the idea of relying upon convex bodies to control maximal rough singular integrals was exploited by Di Plinio, Hyt\"onen and Li \cite{DiPHL}. We shall begin borrowing some definitions from the latter.

Let $1\leq p < \infty$. For every $|\vec{f}| \in L^{p}_{Loc}(\mathbb{R}^d)$ and each cube $Q$ in $\mathbb{R}^{d}$, we define
\begin{eqnarray}
\langle\langle\vec{f}\rangle\rangle_{p,Q}=\left\{ \frac{1}{|Q|}\int_{Q}\vec{f}\varphi dx\,:\,\varphi:Q\rightarrow\mathbb{R},\,\varphi\in B_{L^{p'}(Q)}\right\} \label{eq:2}
\end{eqnarray}
where
 $$B_{L^{p'}(Q)}=\left\{ \phi\in L^{p'}(Q)\,:\,\left(\dfrac{1}{|Q|}\int_Q|\varphi|^{p'}\right)^{\frac{1}{p'}}\leq1\right\}. $$
Note that each set $\langle\langle\vec{f}\rangle\rangle_{p,Q}$ is a compact, convex and symmetric set.

\begin{thm}
\label{Thm:Rough}Let $\Omega\in L^\infty(\mathbb{S}^{d-1})$ with $\int_{\mathbb{S}^{d-1}}\Omega=0$ and $r>1$. Then we have that for each $|\vec{f}|\in L^1(\mathbb{R}^d)$ with compact support and each $|\vec{g}|\in L^r(\mathbb{R}^d)$ there exists a sparse family $\mathcal{S}$ such that 
\begin{equation}
 \int_{\mathbb{R}^d}\left|\langle T_\Omega\vec{f},\vec{g}\rangle\right|\leq c_{n,d}\|\Omega\|_{L^\infty(\mathbb{S}^{d-1})}r' \underset{Q\in \mathcal{S}}{\sum} \langle\langle\vec{f}\rangle\rangle_{1,Q}\langle\langle\vec{g}\rangle\rangle_{r,Q}|Q|\label{eq:Rough}
\end{equation}
\end{thm}
 
Note that this convex body domination result was sketched in \cite[Remark 6.6]{DiPHL}. Here we provide a full proof of this result that has interest on its own.

\begin{thm}
\label{Thm:Horm}Let $r>1$ and let $T$ be a $L^{r'}$-H\"ormander operator.  Then we have that for each $|\vec{f}|\in L^r(\mathbb{R}^d)$ with compact support and each $|\vec{g}|\in L^1(\mathbb{R}^d)$ there exists a sparse family $\mathcal{S}$ such that 
\begin{equation}\label{eq:Horm}
 \int_{\mathbb{R}^d}\left|\langle T\vec{f},\vec{g}\rangle\right|\leq c_{n,d,T} \underset{Q\in \mathcal{S}}{\sum} \langle\langle\vec{f}\rangle\rangle_{r,Q}\langle\langle\vec{g}\rangle\rangle_{1,Q}|Q|
\end{equation}
\end{thm}

Given two convex, compact, symmetric sets $A,B$, the product $AB=\lbrace \langle a,b\rangle :\; a\in A, \; b\in B\rbrace$ is a closed bounded interval. We shall interpret $AB$ as it right endpoint. That will be the case for the products in \eqref{eq:Rough} and \eqref{eq:Horm}.

\begin{rem}
With the available techniques it would be possible to improve \eqref{eq:Horm} to a ``pointwise'' domination result in the spirit of \cite{NPTV}. However since it is not clear that the dependences in our applications derived from the bilinear result can be substantially improved having that result at our disposal we decided to provide just the bilinear domination result for the sake of brevity. 
\end{rem}

\subsection{Proofs of the sparse domination results}\label{sec:MR}
\subsubsection{A convex body domination principle}
This section is devoted to settle the convex body domination principle that we will rely upon in order to settle the results in the preceding section. We shall borrow some ideas and notation from \cite{Le}. Given a sublinear operator $T$ we define the bi-sublinear operator $\mathcal{M}_{T}$ as
$$\mathcal{M}_{T}(f,g)(x)=\underset{Q\ni x}{\sup}\frac{1}{|Q|}\int_{Q}\left|T(f\chi_{\mathbb{R}^{n}\setminus3Q})\right|\left|g\right|.$$
We would also like to recall the John ellipsoid property. If $K\subset \mathbb{R}^{n}$ is a symmetric, closed, convex set, then there exists an ellipsoid $\mathcal{E}_{K}$, such that 
$$\mathcal{E}_{K}\subset K\subset \sqrt{n}\mathcal{E}_{K}$$
where $cA=\lbrace ca\;\;: a\in A\rbrace$.

Before presenting and settling our sparse domination principle we need to borrow a Lemma from \cite[Lemma 6.2]{DiPHL}.
\begin{lem}\label{lem:DiPHL}
Let $f=(f_1,\dots,f_n)\in L^p_{loc}$ suppose that $\mathcal{E}_{\langle\langle f\rangle\rangle_{p,Q}}=B$ where $B$ stands for the unit ball $B=\left\{x\in\mathbb{R}^n : |x|\leq1\right\}$. Then
\[\sup_{j=1,\dots,N}\left(\frac{1}{|Q|}\int_Q|f_j|^p\right)^{\frac{1}{p}}\leq\sqrt{n}\]
\end{lem}

We are now in the position to state and prove our sparse domination principle.
\begin{thm}
\label{Lem:Technical1} Let $1\leq q \leq r$ and $ s\geq 1$. Assume that $T$ is a linear operator of weak type $(q,q)$ and that $\mathcal{M}_T$ maps $L^r\times L^s$ into $L^{\nu,\infty}$ where $\frac{1}{\nu}=\frac{1}{s}+\frac{1}{r}$. Then, for each $\vec{f}$ with compact support such that $|\vec{f}|\in L^{r}(\mathbb{R}^{d})$ and for each $|\vec{g}|\in L^{s}_{Loc}(\mathbb{R}^{d})$, there exists a sparse family $\mathcal{S}$ such that
\[
 \int_{\mathbb{R}^d}\left|\langle T\vec{f},\vec{g}\rangle\right|\leq c_{n,d}\left(\|\mathcal{M}_T\|_{L^r\times L^s\rightarrow L^{\nu,\infty}} +\|T\|_{L^q\rightarrow L^{q,\infty}}\right) \underset{Q\in S}{\sum} \langle\langle\vec{f}\rangle\rangle_{r,Q}\langle\langle\vec{g}\rangle\rangle_{s,Q}|Q| 
\]
\end{thm}
Our argument will rely upon a combination of ideas in \cite{DiPHL, CDiPOu, Le}.
\subsubsection*{Proof of Theorem \ref{Lem:Technical1}}
Fix a cube $Q_0$. We claim that there exists a family of pairwise disjoint cubes $\{P_j\}$ contained in $Q_0$ with $\sum_j|P_j|\leq\frac{1}{2}|Q_0|$
such that
\begin{equation}
\begin{split}
\int_{Q_{0}} \left|\left<T(\vec{f}\chi_{3Q_{0}}),\vec{g}\right>\right|
&\leq c_{n,d}(A_{1}+A_{2}) \langle \langle\vec{f}\rangle\rangle_{r,3Q_0}\langle\langle \vec{g}\rangle\rangle_{s,3Q_0}|3Q_0|\\
&+{\underset{j}{\sum}\int_{P_{j}} \left|\left<T(\vec{f}\chi_{3P_{j}}),\vec{g}\right>\right|}
\end{split}
\label{claim} \end{equation}

We begin observing that for $\vec{f}$ and $\vec{g}$, there exist matrices $M_{1}$, $M_{2} \in GL_{n}(\mathbb{R})$ such that $M_{1}\overset{\vec{\sim}}{f}=\vec{f}$ and $M_{2}\overset{\vec{\sim}}{g}=\vec{g}$ and the John ellipsoid of  $\langle\langle\overset{\vec{\sim}}{f}\rangle\rangle_{r,3Q_{0}}$  and  $\langle\langle\overset{\vec{\sim}}{g}\rangle\rangle_{s,3Q_{0}}$  is the closed unit ball $B$ (see \cite{CDiPOu}). For  $Q_{0}$ let us call
 $$\mathcal{M}_{T,Q_{0}}\left(\overset{\sim}{f_{i}},\overset{\sim}{g_{i}}\right)(x)=\underset{Q\ni x, Q\subset Q_{0}}{\sup}\frac{1}{|Q|}\int_{Q}|T(\overset{\sim}{f_{i}}\chi_{3Q_{0}\setminus3Q})||\overset{\sim}{g_{i}}|dy$$
 Consider the sets
 $$E_{1}^{i}=\lbrace x\in Q_{0} : |T(\overset{\sim}{f_{i}}\chi_{3Q_{0}})(x)|>A_1\langle\overset{\sim}{f_{i}}\rangle_{q,3Q_{0}}\rbrace$$ and 
 $$E_{2}^{i}=\lbrace x\in Q_{0} : |\mathcal{M}_{T,Q_{0}}(\overset{\sim}{f_{i}},\overset{\sim}{g_{i}})(x)|>A_2\langle\overset{\sim}{f_{i}}\rangle_{r,3Q_{0}}\langle\overset{\sim}{g_{i}}\rangle_{s,3Q_{0}}\rbrace$$
 We begin observing that that we can choose $A_1,A_2>0$ such that
 \begin{equation}\label{eq:Omega}|\Omega|\leq \dfrac{1}{2^{d+2}}|Q_{0}|
 \end{equation}
 where $\Omega= E_{1}\cup E_{2}$, $E_{1}=\underset{i=1}{\overset{n}{\bigcup}}E_{1}^{i}$ and $E_{2}=\underset{i=1}{\overset{n}{\bigcup}}E_{2}^{i}$.
 
 First we note that
 \begin{align*}
 |E^{i}_{1}|&=|x\in Q_{0} :|T(\overset{\sim}{f_{i}}\chi_{3Q_{0}})(x)|>A_1\langle\overset{\sim}{f_{i}}\rangle_{q,3Q_{0}}\rbrace|\leq \frac{1}{(A_1\langle\overset{\sim}{f_{i}}\rangle_{q,3Q_{0}})^{q}}\|T\|^{q}_{L^{{q}}\rightarrow L^{{q, \infty}}}\|\overset{\sim}{f_{i}}\|^{q}_{L^{q}(3Q_{0})}\\
 &\leq \frac{1}{A_1^{q}\frac{1}{|3Q_{0}|}\int_{3Q_{0}} |\overset{\sim}{f_{i}}|^{q}dx}\|T\|^{q}_{L^{{q}}\rightarrow L^{{q, \infty}}}|3Q_{0}|\dfrac{1}{|3Q_{0}|}\int_{3Q_{0}} |\overset{\sim}{f_{i}}|^{q}dx\\
 &=\dfrac{1}{A_1^{q}}\|T\|^{q}_{L^{{q}}\rightarrow L^{{q, \infty}}}3^{d}|Q_{0}|    
 \end{align*}
Hence, choosing $A_1=\|T\|_{L^{{q}}\rightarrow L^{{q, \infty}}}3^{\frac{d}{q}}2^{\frac{d+3}{q}}n^{\frac{1}{q}}$ we have that $|E_1|\leq\frac{1}{2^{d+3}}|Q_0|$.

Next, we observe that 
 \begin{align*}|E_{2}^{i}|&=|\lbrace x\in Q_{0} : |\mathcal{M}_{T,Q_{0}}(\overset{\sim}{f_{i}},\overset{\sim}{g_{i}})(x)|>A_{2}\langle\overset{\sim}{f_{i}}\rangle_{r,3Q_{0}}\langle\overset{\sim}{g_{i}}\rangle{s,3Q_{0}}\rbrace|\\
 &\leq \dfrac{1}{\left( A_{2}\langle\overset{\sim}{f_{i}}\rangle_{r,3Q_{0}}\langle\overset{\sim}{g_{i}}\rangle_{s,3Q_{0}}\right)^{\nu}}\|\mathcal{M}_{T}\|^{\nu}_{ L^{r}\times L^{s}\rightarrow L^{\nu,\infty}}\|\overset{\sim}{f_{i}}\|_{L^{r}(3Q_{0})}^{\nu}\|\overset{\sim}{g_{i}}\|_{L^{s}(3Q_{0})}^{\nu}\\
 &\leq \dfrac{\|\mathcal{M}_{T,Q_{0}}\|^{\nu}_{ L^{r}\times L^{s}\rightarrow L^{\nu,\infty}}}{ A_{2}^{\nu}}|3Q_0|\leq \dfrac{\|\mathcal{M}_{T,Q_{0}}\|^{\nu}_{ L^{r}\times L^{s}\rightarrow L^{\nu,\infty}}}{ A_{2}^{\nu}}3^d|Q_0|
 \end{align*}
and choosing $A_2=\|\mathcal{M}_{T,Q_{0}}\|_{ L^{r}\times L^{s}\rightarrow L^{\nu,\infty}}3^{\frac{d}{\nu}}2^{\frac{d+3}{\nu}}n^{\frac{1}{\nu}}$ we have that $|E_2|\leq\frac{1}{2^{d+3}}|Q_0|$.
Combining the estimates above \eqref{eq:Omega} readily follows.

Now we form the Calder\'on-Zygmund decomposition with respect to $Q_0$ of $\chi_\Omega$ at height $\frac{1}{2^{d+1}}$. We obtain a family of pairwise disjoint cubes $P_{j}\in\mathcal{D}(Q_{0})$, such that
\begin{align}
\dfrac{1}{2^{d+1}}|P_{j}|&\leq |P_{j}\cap \Omega|\leq \dfrac{1}{2}|P_{j}|\\
|\Omega\backslash \cup_{j}P_{j}|&=0 \\
\sum_{j}|P_{j}|&\leq \frac{1}{2}|Q_{0}|\label{eq:sparse}\\
P_{j}\cap \Omega^{c}&\neq \emptyset
\end{align} 

Having that family of cubes at our disposal we continue our argument as follows. 
\begin{align*}
\int_{Q_{0}} \left|\left<T(\vec{f}\chi_{3Q_{0}}),\vec{g}\right>\right|&=\int_{Q_{0}\setminus\cup P_{j}} \left|\left<T(\vec{f}\chi_{3Q_{0}}),\vec{g}\right>\right|+ \underset{j}{\sum}\int_{P_{j}} \left|\left<T(\vec{f}\chi_{3Q_{0}}),\vec{g}\right>\right|\\
&\leq {\int_{Q_{0}\setminus\cup P_{j}} \left|\left<T(\vec{f}\chi_{3Q_{0}}),\vec{g}\right>\right|}+ {\underset{j}{\sum}\int_{P_{j}} \left|\left<T(\vec{f}\chi_{3Q_{0}\backslash 3P_{j}}),\vec{g}\right>\right|}+{\underset{j}{\sum}\int_{P_{j}} \left|\left<T(\vec{f}\chi_{3P_{j}}),\vec{g}\right>\right|}\\
&=I+II+III
\end{align*}

First we deal with $I$.

\begin{align*}
\int_{Q_{0}\setminus\cup P_{j}} \left|\left<T(\vec{f}\chi_{3Q_{0}}),\vec{g}\right>\right|&=\int_{Q_{0}\setminus\cup P_{j}} \left|\left<T(M_{1}\overset{\vec{\sim}}{f}\chi_{3Q_{0}}),M_{2}\overset{\vec{\sim}}{g}\right>\right| \leq \int_{Q_{0}\setminus \cup P_{j}} \left|\left< M_{1}T(\overset{\vec{\sim}}{f}\chi_{3Q_{0}}),M_{2}\overset{\vec{\sim}}{g}\right>\right|\\
&=\int_{Q_{0}\setminus\cup P_{j}} \left|\underset{i,k_{1},k_{2}=1}{\overset{n}{\sum}}M^{ik_{1}}_{1}M^{ik_{2}}_{2}T(\overset{\sim}{f_{i}}\chi_{3Q_{0}})\overset{\sim}{g_{i}}\right|\\
&\leq \underset{1\leq k_{1},k_{2}\leq n}{\sup}\left|\underset{i=1}{\overset{n}{\sum}} M^{ik_{1}}_{1}M^{ik_{2}}_2\right|\underset{i=1}{\overset{n}{\sum}}\int_{Q_{0}\setminus\cup P_{j}}\left|T(\overset{\sim}{f_{i}}\chi_{3Q_{0}})\overset{\sim}{g_{i}}\right|
\end{align*}

Since $|\Omega \backslash \cup_{j}P_{j}|=0$, we can continue as follows
\begin{align*}
&\underset{1\leq k_{1},k_{2}\leq n}{\sup}\left|\underset{i=1}{\overset{n}{\sum}} M^{ik_{1}}_{1}M^{ik_{2}}_2\right|\underset{i=1}{\overset{n}{\sum}}\int_{Q_{0}\setminus\cup P_{j}}\left|T(\overset{\sim}{f_{i}}\chi_{3Q_{0}})\overset{\sim}{g_{i}}\right|\\
&\leq \underset{1\leq k_{1},k_{2}\leq n}{\sup}\left|\underset{i=1}{\overset{n}{\sum}} M^{ik_{1}}_{1}M^{ik_{2}}_2\right|\underset{i=1}{\overset{n}{\sum}}A_{1}\langle\overset{\sim}{f_{i}}\rangle_{q,3Q_{0}}\int_{Q_{0}}| \overset{\sim}{g_{i}}|\\
&\leq \underset{1\leq k_{1},k_{2}\leq n}{\sup}\left|\underset{i=1}{\overset{n}{\sum}} M^{ik_{1}}_{1}M^{ik_{2}}_2\right|\underset{i=1}{\overset{n}{\sum}}A_{1}\langle\overset{\sim}{f_{i}}\rangle_{q,3Q_{0}}\int_{3Q_{0}}| \overset{\sim}{g_{i}}|\\
&\leq \underset{1\leq k_{1},k_{2}\leq n}{\sup}\left|\underset{i=1}{\overset{n}{\sum}} M^{ik_{1}}_{1}M^{ik_{2}}_2\right|\underset{i=1}{\overset{n}{\sum}}A_{1}\langle\overset{\sim}{f_{i}}\rangle_{q,3Q_{0}}\langle\overset{\sim}{g_{i}}\rangle_{3Q_{0}}3^{d}|Q_{0}|.\end{align*}
By Lemma \ref{lem:DiPHL} we have that $\underset{i=1,\dots,N}{\sup}\langle\overset{\sim}{f_{i}}\rangle_{q,3Q_{0}}\leq \sqrt{n}$ and also that  $\underset{i=1,\dots,N}{\sup}\langle\overset{\sim}{g_{i}}\rangle_{q,3Q_{0}}\leq \sqrt{n}$. Therefore, the last part of the right term of the inequality is bounded by a dimensional constant, namely,
$$\underset{1\leq k_{1},k_{2}\leq n}{\sup}\left|\underset{i=1}{\overset{n}{\sum}} M^{ik_{1}}_{1}M_2^{ik_{2}}\right|\underset{i=1}{\overset{n}{\sum}}A_{1}\langle\overset{\sim}{f_{i}}\rangle_{q,3Q_{0}}\langle\overset{\sim}{g_{i}}\rangle_{3Q_{0}}3^{d}|Q_{0}|\leq\underset{1\leq k_{1},k_{2}\leq n}{\sup}\left|\underset{i=1}{\overset{n}{\sum}} M^{ik_{1}}_{1}M_2^{ik_{2}}\right|A_{1}C_{n,d}|Q_{0}|.$$
It remains to provide an estimate for 
 $\underset{1\leq k_{1},k_{2}\leq n}{\sup}\left|\underset{i=1}{\overset{n}{\sum}} M^{i,k_{1}}_{1}M^{i,k_{2}}\right|$.
 
We claim that
\begin{equation}\label{eq:M1M2}
\left|\underset{i=1}{\overset{n}{\sum}} M^{i,k_{1}}_{1}M_2^{i,k_{2}}\right|=|\langle M_{1}{e}_{k_{1}},M_{2}{e}_{k_{2}}\rangle|\leq \langle\langle \vec{f}\rangle\rangle_{r,3Q_{0}}\langle \langle\vec{g}\rangle\rangle_{s,3Q_{0}}\qquad 1\leq k_{1},k_{2}\leq n
\end{equation}
where ${e}_{k}$ is the $k$-th coordinate  vector. 
Indeed, fix  $ 1\leq k_{1},k_{2}\leq n$.
Since ${e}_{k_{1}}$ belongs to the unit ball $ B=\langle\langle\overset{\vec{\sim}}{f}\rangle\rangle_{r,3Q_{0}}$ there exists $\varphi_{1}\in B_{L^{r'}(3Q_0)}$ such that
$$e_{k_{1}}=\dfrac{1}{|3Q_{0}|}\int_{3Q_{0}} \overset{\vec{\sim}}{f}\varphi_{1}$$
Therefore,
$$M_{1}e_{k_{1}}=\dfrac{1}{|3Q_{0}|}\int_{3Q_{0}} M_{1} \overset{\vec{\sim}}{f_{i}}\varphi_{1}=\dfrac{1}{|3Q_{0}|}\int_{3Q_{0}} \vec{f}\varphi_{1}\in  \langle\langle \vec{f}\rangle\rangle_{r,3Q_{0}}.$$
Analogously for $e_{k_{2}}$, we have that 
$M_{2}e_{k_{2}}
\in \langle\langle \vec{g}\rangle\rangle_{s,3Q_{0}}$
and hence \eqref{eq:M1M2} holds.

Combining the estimates above we have that
$$ I=\int_{Q_{0}\setminus\cup P_{j}} \left|\left<T(M_{1}\overset{\vec{\sim}}{f}\chi_{3Q_{0}}),M_{2}\overset{\vec{\sim}}{g}\right>\right|\leq A_{1}C_{n,d}|Q_{0}|\langle \langle\vec{f}\rangle\rangle_{r,3Q_{0}}\langle\langle \vec{g}\rangle\rangle_{s,3Q_{0}}$$

For $II$ we begin arguing as we did for $I$. Since  $\vec{f}=M_{1}\overset{\vec{\sim}}{f}$ and $\vec{g}=M_{2}\overset{\vec{\sim}}{g}$ we have that
\begin{align*}
&\underset{j}{\sum}\int_{P_{j}} \left|\left<T(\vec{f}\chi_{3Q_{0}\backslash 3P_{j}}),\vec{g}\right>\right|=\underset{j}{\sum}\int_{P_{j}} \left|\left<T(M_{1}\overset{\vec{\sim}}{f}\chi_{3Q_{0}\backslash 3P_{j}}),M_{2}\overset{\vec{\sim}}{g}\right>\right|\\
&\leq \underset{j}{\sum} \int_{P_{j}}\left|\left< M_{1}T(\overset{\vec{\sim}}{f}\chi_{3Q_{0}\setminus 3P_{j}}),M_{2}\overset{\vec{\sim}}{g}\right>\right|= \underset{j}{\sum}\int_{P_{j}} \left|\underset{i,k_{1},k_{2}=1}{\overset{n}{\sum}}M^{ik_{1}}_{1}M^{ik_{2}}_{2}T(\overset{\sim}{f_{i}}\chi_{3Q_{0}\setminus 3P_{j}})\overset{\sim}{g_{i}}\right|\\
&\leq \underset{1\leq k_{1},k_{2}\leq n}{\sup}\left|\underset{i=1}{\overset{n}{\sum}} M^{ik_{1}}_{1}M^{ik_{2}}_2\right| \underset{j}{\sum}\underset{i=1}{\overset{n}{\sum}}\int_{P_{j}}\left|T(\overset{\sim}{f_{i}}\chi_{3Q_{0}\setminus 3P_{j}})\overset{\sim}{g_{i}}\right|
\end{align*}
At this point since $P_{j}\cap \Omega^{c}\neq \emptyset$ and also $\sum_{j}|P_{j}|\leq \frac{1}{2}|Q_{0}|$ we have that 
$$\underset{j}{\sum} \int_{P_{j}}\left|T(\overset{\sim}{f_{i}}\chi_{3Q_{0}\setminus 3P_{j}})\overset{\sim}{g_{i}}\right|\leq c_d\underset{j}{\sum}A_{2}\langle\overset{\sim}{f_{i}}\rangle_{r,3Q_{0}}\langle\overset{\sim}{g_{i}}\rangle_{s,3Q_{0}}|P_{{j}}|\leq c_d\dfrac{A_2}{2}\langle\overset{\sim}{f_{i}}\rangle_{r,3Q_{0}}\langle\overset{\sim}{g_{i}}\rangle_{s,3Q_{0}}|Q_{0}|.$$
Arguing as above by Lemma \ref{lem:DiPHL} the right term of the inequality above is bounded by a dimensional constant. 
Combining the estimates above
$$\underset{j}{\sum}\int_{P_{j}} \left|\left<T(M_{1}\overset{\vec{\sim}}{f}\chi_{3Q_{0}\backslash 3P_{j}}),M_{2}\overset{\vec{\sim}}{g}\right>\right|\leq \underset{1\leq k_{1},k_{2}\leq n}{\sup}\left|\underset{i=1}{\overset{n}{\sum}} M^{ik_{1}}_{1}M_2^{ik_{2}}\right|A_{2}C'_{n,d}|Q_{0}|$$
which combined with \eqref{eq:M1M2} yields that
$$II\leq A_{2}C'_{n,d}\langle\langle \vec{f}\rangle\rangle_{r,3Q_{0}}\langle\langle \vec{g}\rangle\rangle_{s,3Q_{0}}|Q_{0}|.$$
Taking into account the estimates for $I$, $II$ and the properties of the family $\{P_j\}$ the claim \eqref{claim} at the beginning of the proof is settled.

It is not hard to check that iterating the claim leads to the construction of a family of cubes $\mathcal{F}$ contained in $Q_0$ which is $\frac{1}{2}$-sparse and such that
\begin{equation}\label{eq:local}
\int_{Q_0} \left|\left<T(\vec{f}\chi_{3Q_{0}}),\vec{g}\right>\right|\leq c_{n,d}(A_{1}+A_{2})\underset{Q\subset \mathcal{F}}{\sum}  \langle\langle \vec{f}\rangle\rangle_{r,3Q_0}\langle\langle \vec{g}\rangle\rangle_{s,3Q_0}|3Q_0|
\end{equation}
Relying upon the preceding estimate we show now how to end the proof. Take a partition of  $\mathbb{R}^{n}$ by cubes  $R_{j}$ such that  supp($\vec{f})\subset 3R_{j}$ for each $j$. For example, take a cube $Q_{0}$ such that supp($\vec{f})\subset Q_{0}$ and cover $3Q_{0}\backslash Q_{0}$ by $3^{n}-1$  congruent cubes $R_{j}$. Each of them satisfies $Q_{0}\subset 3R_{j}$ Next, in the same way cover $9Q_{0}\backslash 3Q_{0}$ and so on. The union of resulting cubes, including  $Q_{0}$, will satisfy the desired property. Therefore, applying \eqref{eq:local} to each $R_j$ as follows
 $$\int_{\mathbb{R}^d} \left|\left<T(\vec{f}\chi_{3Q_{0}}),\vec{g}\right>\right|=
 \sum_j\int_{R_j}\left|\left<T(\vec{f}\chi_{3R_j}),\vec{g}\right>\right| \leq c_{n,d}(A_{1}+A_{2}) \underset{Q\in\bigcup_{j} \mathcal{F}_{j}}{\sum}  \langle\langle \vec{f}\rangle\rangle_{r,Q}\langle\langle \vec{g}\rangle\rangle_{s,Q}|Q|.$$
  
Note that the family  $\bigcup_{j} \mathcal{F}_{j}$ is $\frac{1}{2}$-sparse as a disjoint union of $\frac{1}{2}$-sparse families. Hence, setting $\mathcal{S}= \lbrace 3Q: Q \in \cup_{j} \mathcal{F}_{j}\rbrace$, we obtain that $\mathcal{S} $ is $\frac{1}{2\cdot3^{n}}$-sparse. This ends the proof of the Theorem.
\subsubsection{Proof of Theorem \ref{Thm:Rough}}
Given $1\le p\le \infty$, we define the maximal operator ${\mathcal M}_{p,T}$ by
$${\mathcal M}_{p,T}f(x)=\sup_{Q\ni x}\left(\frac{1}{|Q|}\int_Q|T(f\chi_{{\mathbb R}^n\setminus 3Q})|^pdy\right)^{1/p}$$

Note that in \cite{Le} it was shown that for every $p\geq 1$,
\begin{equation}\label{eq:MpTRoughLerner}
\|{\mathcal M}_{p,T_{\Omega}}\|_{L^1\to L^{1,\infty}}\le c\|\Omega\|_{L^\infty(\mathbb{S}^{n-1})}p
\end{equation}
Observe that taking that into account, we have that
\begin{align*}
\mathcal{M}_{T_{\Omega}}(f,g)(x) & =\sup_{Q\ni x}\frac{1}{|Q|}\int_{Q}\left|T_{\Omega}(f\chi_{\mathbb{R}^{n}\setminus3Q})\right|\left|g\right|\\
 &\leq\sup_{Q\ni x}\left(\frac{1}{|Q|}\int_{Q}\left|T_{\Omega}(f\chi_{\mathbb{R}^{n}\setminus3Q})\right|^{r'}\right)^{\frac{1}{r'}}M_{r}(g) =\mathcal{M}_{T,r'}(f)M_{r}(g)
\end{align*}
By Hölder inquality for weak type spaces, combined with \eqref{eq:MpTRoughLerner}
\begin{equation}\label{eq:MTRough}
\|\mathcal{M}_{T_{\Omega}}(f,g)\|_{L^{\frac{r}{r+1},\infty}}\lesssim\|\mathcal{M}_{T_{\Omega},r'}(f)\|_{L^{1,\infty}}\|M_{r}g\|_{L^{r,\infty}}\l\lesssim r' \|\Omega\|_{L^\infty(\mathbb{S}^{d-1})}\|f\|_{L^1}\|g\|_{L^r}
\end{equation}
Taking into account  that
\[\|T_{\Omega}\|_{L^1\rightarrow L^{1,\infty}}\leq c_d\|\Omega\|_{L^\infty(\mathbb{S}^{d-1})}\]
and \label{eq:MTRough}${M}_{T_{\Omega}}$ Theorem \ref{Thm:Rough} readily follows from Theorem \ref{Lem:Technical1}.

\subsubsection{Proof of Theorem \ref{Thm:Horm}}
To settle the theorem it suffices to apply Theorem  \ref{Lem:Technical1} combined with the fact that
\begin{equation}\label{eq:MTHorm}
\|\mathcal{M}_{T}(f,g)\|_{L^{\frac{r}{r+1},\infty}}\leq c_{n,T} \|f\|_{L^r} \|g\|_{L^1}
\end{equation}
and that $T$ is of weak type $(1,1)$ which is well known. Hence it remains to settle the latter. Note that 
$$\mathcal{M}_{T}(f,g)(x)=\sup_{Q\ni x}\frac{1}{|Q|}\int_{Q}\left|T(f\chi_{\mathbb{R}^{n}\setminus3Q})\right|\left|g\right|
\leq (Mg) \sup_{Q\ni x}\|T(f\chi_{\mathbb{R}^{n}\setminus3Q})\|_{L^\infty(Q)}=Mg \mathcal{M}_{T,\infty}(f)
$$
Then we have that by H\"older inequality for weak spaces, $$\|\mathcal{M}_{T}(f,g)\|_{L^{\frac{r}{r+1},\infty}}\leq \|\mathcal{M}_{T,\infty}(f)\|_{L^{r,\infty}}\|Mg\|_{L^{1,\infty}}$$

In \cite{Li}, Li showed that $\|\mathcal{M}_{T,\infty}(f)\|_{L^{r,\infty}}\leq c_{n,T} \|f\|_{L^r}$ . This fact combined with the well-known endpoint estimate for the maximal function, yields \eqref{eq:MTHorm}.

\section{Some further definitions and Lemmata}\label{sec:defsLemmata}
We recall that norms on $\mathbb{R}^n$ can be represented by positive definite self-adjoint matrices, namely, if $\rho:\mathbb{R}^n\rightarrow\mathbb{R}$ is a norm, then there exists  a positive definite self-adjoint matrix $A$ such that 
$|Ae|\simeq\rho(e)$.  We remit the reader to \cite{V} for more details.

This fact is particularly useful when dealing with matrix weights. Given a matrix weight $W$ and $p\geq1$ we will call $\mathcal{W}_{p,Q}$ a matrix such that 
\[|\mathcal{W}_{p,Q}\vec{e}|\simeq\left(\frac{1}{|Q|}\int_Q|W^\frac{1}{p}(x)\vec{e}|^pdx\right)^{\frac{1}{p}}\]
and if $p>1$ we will call $\mathcal{W}_{p',Q}$ a matrix such that 
\[|\mathcal{W}_{p',Q}\vec{e}|\simeq\left(\frac{1}{|Q|}\int_Q|W^{-\frac{1}{p'}}(x)\vec{e}|^{p'}dx\right)^{\frac{1}{p'}}.\]

Relying upon this definition we observe that the $A_p$ condition can be expressed in terms of reducing matrices. This follows from the fact that
\[\frac{1}{|Q|}\int_{Q}\left(\frac{1}{|Q|}\int\left|W^{\frac{1}{p}}(x)W^{-\frac{1}{p}}(y)\right|_{op}^{p'}dy\right)^{\frac{p}{p'}}dx\simeq|\mathcal{W}_{p,Q}\mathcal{W}_{p',Q}|_{op}^p\]
for $p>1$, and
\[\frac{1}{|Q|}\int_Q|W(x)W^{-1}(y)|_{op}dx\simeq|\mathcal{W}_{1,Q}W^{-1}(y)|_{op}\]
for $p=1$.

Another property related to matrix weights that will be fundamental for us is the reverse H\"older property. It was shown in \cite{HPAinfty} (see \cite{HPR} for an alternative proof) that if $w\in A_\infty$ then 
\[
\left(\frac{1}{|Q|}\int_Qw^r(x)\right)^{\frac{1}{r}}\leq 2 \frac{1}{|Q|}\int_Q w(x)
\]
where $r=1+\frac{1}{2^{d+11}[w]_{A_\infty}}$. 

Recall that if $W\in A_p$ then we have that $|W^\frac{1}{p}\vec{e}|^p$ are scalar $A_p$ weights uniformly on $\vec{e}$ and consequently, $|W^\frac{1}{p}\vec{e}|^p$ are scalar $A_\infty$ weights, with scalar $A_\infty$ constants uniformly controlled by $[W]_{A_p}$. This fact allows to make sense of \eqref{eq:AinftyMat}.

A consequence of those definitions is the following Reverse H\"older inequality that we will repeatedly throughout the remainder of the paper. 
\begin{lem}
Let $A$ by a self-adjoint positive definite matrix and let $1\leq p<\infty$. Assume that $W\in A_{\infty,p}^{sc}$. Then, if $r\leq1+\frac{1}{2^{d+11}[W]_{A_{\infty,p}^{sc}}}$ we have that
\[
\left(\frac{1}{|Q|}\int_Q|W^\frac{1}{p}(x)A|_{op}^r\right)^{\frac{1}{r}}\lesssim \frac{1}{|Q|}\int_Q|W^\frac{1}{p}(x)A|_{op}
\]
\end{lem}
\begin{proof}
We fix some orthonormal basis $\{e_j\}$ on $\mathbb{R}^n$. Taking into account $|W^\frac{1}{p}\vec{e}|^p$ satisfies the scalar reverse H\"older inequality uniformly on $\vec{e}$ for $r$ due to the fact that $W\in A_{\infty,p}^{sc}$ we have that
\begin{align*}
\left(\frac{1}{|Q|}\int_Q|W^\frac{1}{p}A|_{op}^r\right)^{\frac{1}{r}}\lesssim\sum_{j=1}^n\left(\frac{1}{|Q|}\int_Q|W^\frac{1}{p}Ae_j|^r\right)^{\frac{1}{r}}\leq2 \sum_{j=1}^n\frac{1}{|Q|}\int_Q|W^\frac{1}{p}Ae_j|\lesssim \frac{1}{|Q|}\int_Q|W^\frac{1}{p}A|_{op}
\end{align*}
\end{proof}
\begin{rem}
Note that given two positive definite self-adjoint matrices, $|AB|_{op}\simeq|BA|_{op}$, the estimate in the preceding lemma holds as well reversing the order of the matrices involved.
\end{rem}

Now we gather some Lemmata that will be useful throughout the remainder of the paper. The first of them will help us to settle strong type estimates.
\begin{lem}\label{lem:KeyAp}
Let $p,r,s\geq 1$ and let $W$ be a weight. For each $\eta$-sparse family, 
\allowdisplaybreaks
\begin{align*}
\underset{Q\in \mathcal{S}}{\sum}& \langle\langle W^{-\frac{1}{p}}\vec{h}\rangle\rangle_{r,Q}\langle\langle W^{\frac{1}{p}}\vec{g}\rangle\rangle_{s,Q}|Q|\\
&\leq \frac{1}{\eta}
\sup_Q\left|\mathcal{V}_{Q}\mathcal{U}_{Q}\right|
\|M_{\mathcal{V},W^{-\frac{1}{p}},r}(\vec{h})\|_{L^p}\|M_{\mathcal{U},W^{\frac{1}{p}},s}\|_{L^{p'}}\|\vec{g}\|_{L^{p'}(\mathbb{R}^d;\mathbb{R}^n)}\\
&\leq \frac{1}{\eta}
\sup_Q\left|\mathcal{V}_{Q}\mathcal{U}_{Q}\right|
\|M_{\mathcal{V},W^{-\frac{1}{p}},r}\|_{L^p}\|M_{\mathcal{U},W^{\frac{1}{p}},s}\|_{L^{p'}}\|\vec{h}\|_{L^p(\mathbb{R}^d;\mathbb{R}^n)}\|\vec{g}\|_{L^{p'}(\mathbb{R}^d;\mathbb{R}^n)}.
\end{align*}
where
\allowdisplaybreaks
\begin{align*}M_{\mathcal{V},W^{-\frac{1}{p}},r}(\vec{h})(z)&=\sup_{x\in Q}\left(\frac{1}{|Q|}\int_{Q}|(\mathcal{V}_{Q})^{-1}W^{-\frac{1}{p}}(x)\vec{h}(x)|^{r}dx\right)^{\frac{1}{r}}\\
M_{\mathcal{U},W^{\frac{1}{p}},s}(\vec{g})(z)&=\sup_{x\in Q} \left(\frac{1}{|Q|}\int_{Q}|(\mathcal{U}_{Q})^{-1}W^{\frac{1}{p}}(x)\vec{g}(x)|^{s}dx\right)^{\frac{1}{s}}
\end{align*}
and $\{\mathcal{U}_Q\}_Q$ $\{\mathcal{V}_Q\}_Q$ are families of self-adjoint positive definite matrices.
\end{lem}
\begin{proof}
First we observe that taking into account that each $\mathcal{U}_Q$ and each $\mathcal{V}_Q$ are self-adjoint positive definite matrices,
\begin{align*}\langle\langle &W^{-\frac{1}{p}}\vec{h}(x)\rangle\rangle_{r,Q}\langle\langle W^{\frac{1}{p}}\vec{g}(x)\rangle\rangle_{s,Q}\\
&=\underset{\underset{\|\psi\|_{L^{s'}}\leq 1}{\|\varphi\|_{L^{r'}}\leq 1}}{sup} \left\lbrace \left\langle \frac{1}{|Q|}\int_{Q} W^{-\frac{1}{p}}(x)\vec{h}(x)\varphi(x) dx,\frac{1}{|Q|}\int_{Q} W^{\frac{1}{p}}(x)\vec{g}(x)\psi(x) dx\right\rangle \right\rbrace\\
&=\underset{\underset{\|\psi\|_{L^{s'}}\leq 1}{\|\varphi\|_{L^{r'}}\leq 1}}{sup} \left\lbrace \left\langle  \frac{1}{|Q|}\int_{Q}\mathcal{V}_{Q}(\mathcal{V}_{Q})^{-1}  W^{-\frac{1}{p}}(x)\vec{h}(x)\varphi(x) dx,\frac{1}{|Q|}\int_{Q} \mathcal{U}_{Q}(\mathcal{U}_{Q})^{-1}W^{\frac{1}{p}}(x)\vec{g}(x)\psi(x) dx\right\rangle \right\rbrace
\\
&=\underset{\underset{\|\psi\|_{L^{s'}}\leq 1}{\|\varphi\|_{L^{r'}}\leq 1}}{sup} \left\lbrace \left\langle  \frac{1}{|Q|}\int_{Q}\mathcal{U}_{Q}\mathcal{V}_{Q}(\mathcal{V}_{Q})^{-1}  W^{-\frac{1}{p}}(x)\vec{h}(x)\varphi(x) dx,\frac{1}{|Q|}\int_{Q} (\mathcal{U}_{Q})^{-1}W^{\frac{1}{p}}(x)\vec{g}(x)\psi(x) dx\right\rangle\right\rbrace
\\
&\leq\underset{\underset{\|\psi\|_{L^{s'}}\leq1}{\|\varphi\|_{L^{r'}}\leq1}}{\sup}\left\lbrace \left(\frac{1}{|Q|}\int_{Q}|\mathcal{U}_{Q}\mathcal{V}_{Q}(\mathcal{V}_{Q})^{-1}W^{-\frac{1}{p}}(x)\vec{h}(x)\varphi(x)|dx\right)\left(\frac{1}{|Q|}\int_{Q}|(\mathcal{U}_{Q})^{-1}W^{\frac{1}{p}}(x)\vec{g}(x)\psi(x)|dx\right)\right\rbrace \\&\leq\sup_{Q}|\mathcal{U}_{Q}\mathcal{V}_{Q}|_{op}\underset{\underset{\|\psi\|_{L^{s'}}\leq1}{\|\varphi\|_{L^{r'}}\leq1}}{\sup}\left\lbrace \left(\frac{1}{|Q|}\int_{Q}|(\mathcal{V}_{Q})^{-1}W^{-\frac{1}{p}}(x)\vec{h}(x)\varphi(x)|dx\right)\left(\frac{1}{|Q|}\int_{Q}|(\mathcal{U}_{Q})^{-1}W^{\frac{1}{p}}(x)\vec{g}(x)\psi(x)|dx\right)\right\rbrace \\&\leq\sup_{Q}|\mathcal{U}_{Q}\mathcal{V}_{Q}|_{op}\left(\frac{1}{|Q|}\int_{Q}|(\mathcal{V}_{Q})^{-1}W^{-\frac{1}{p}}(x)\vec{h}(x)|^{r}dx\right)^{\frac{1}{r}}\left(\frac{1}{|Q|}\int_{Q}|(\mathcal{U}_{Q})^{-1}W^{\frac{1}{p}}(x)\vec{g}(x)|^{s}dx\right)^{\frac{1}{s}}\\
&\leq\sup_{Q}|\mathcal{U}_{Q}\mathcal{V}_{Q}|_{op}\inf_{z\in Q}M_{\mathcal{V},W^{-\frac{1}{p}},r}(\vec{h})(z)\inf_{z\in Q}M_{\mathcal{U},W^{\frac{1}{p}},s}(\vec{g})(z)
\end{align*}
Taking this into account, 
\begin{align*}
 & \underset{Q\in\mathcal{S}}{\sum}\langle\langle W^{-\frac{1}{p}}\vec{h}\rangle\rangle_{r,Q}\langle\langle W^{\frac{1}{p}}\vec{g}\rangle\rangle_{s,Q}|Q|\\
 & \leq\frac{1}{\eta}\sup_{Q}\left|\mathcal{U}_{Q}\mathcal{V}_{Q}\right|_{op}\underset{Q\in\mathcal{S}}{\sum}\inf_{z\in Q}M_{\mathcal{V},W^{-\frac{1}{p}},r}(\vec{h})(z)\inf_{z\in Q}M_{\mathcal{U},W^{\frac{1}{p}},s}(\vec{h})(z)|E_{Q}|\\
 & \leq\frac{1}{\eta}\sup_{Q}\left|\mathcal{U}_{Q}\mathcal{V}_{Q}\right|_{op}\int_{\mathbb{R}^{d}}\inf_{z\in Q}M_{\mathcal{V},W^{-\frac{1}{p}},r}(\vec{h})(x)M_{\mathcal{U},W^{\frac{1}{p}},s}(\vec{g})(x)dx\\
 & \leq\frac{1}{\eta}\sup_{Q}\left|\mathcal{U}_{Q}\mathcal{V}_{Q}\right|_{op}\|M_{\mathcal{V},W^{-\frac{1}{p}},r}(\vec{h})\|_{L^{p}}\|M_{\mathcal{U},W^{\frac{1}{p}},s}(\vec{g})\|_{L^{p'}}
\end{align*}
from which the desired result readily follows.\end{proof}

The following Lemma will allow us to reduce bumped weight conditions to $A_p$ type conditions.
\begin{lem}\label{lem:ApFromRH}
Let $q,r,s>1$. Assume that 
\begin{align*}|\mathcal{V}_Q\vec{e}|&\simeq\left(\frac{1}{|Q|}\int_Q|W^{-\frac{1}{q}}(x)\vec{e}|^{q'r}\right)^\frac{1}{rq'}\\
|\mathcal{U}_Q\vec{e}|&\simeq\left(\frac{1}{|Q|}\int_Q|W^{\frac{1}{q}}(x)\vec{e}|^{qs}\right)^\frac{1}{qs}
\end{align*}
for every $\vec{e}\in\mathbb{R}^n$ and that 
\begin{align*}\left(\frac{1}{|Q|}\int_Q|W^{-\frac{1}{q}}(x)\vec{e}|^{q'r}\right)^\frac{1}{rq'}&\lesssim\left(\frac{1}{|Q|}\int_Q|W^{-\frac{1}{q}}(x)\vec{e}|^{q'}\right)^\frac{1}{q'}\\
\left(\frac{1}{|Q|}\int_Q|W^{\frac{1}{q}}(x)\vec{e}|^{qs}\right)^\frac{1}{qs}& \lesssim\left(\frac{1}{|Q|}\int_Q|W^{\frac{1}{q}}(x)\vec{e}|^{q}\right)^\frac{1}{q}
\end{align*}
for every $\vec{e}\in\mathbb{R}^n$. Then
\[|\mathcal{V}_Q\mathcal{U}_Q\vec{e}|\lesssim|\mathcal{W}_{Q,q}\mathcal{W}'_{Q,q}\vec{e}|.\]
\end{lem}
\begin{proof}
Note that, taking into account the reverse H\"older inequality in the hypothesis,
\begin{align*}|\mathcal{V}_Q\mathcal{U}_Q\vec{e}|&\simeq \left(\dashint_{Q}|W^{-\frac{1}{q}}(x)\mathcal{U}_Q\vec{e}|^{rq'}\right)^{\frac{1}{rq'}}
\lesssim \left(\dashint_{Q}|W^{\frac{-1}{q}}(x)\mathcal{U}_Q\vec{e}|^{q'}\right)^{\frac{1}{q'}}\\
&\simeq |\mathcal{W}'_{Q,q}\mathcal{U}_Q\vec{e}|.
\end{align*}
Hence \[|\mathcal{V}_Q\mathcal{U}_Q|_{op}\lesssim  |\mathcal{W}'_{Q,q}\mathcal{U}_Q|_{op}.\]
Now observe that  $|\mathcal{W}'_{Q,q}\mathcal{U}_Q|_{op}=|\mathcal{U}_Q\mathcal{W}'_{Q,q}|_{op}$. Then, again by the reverse H\"older inequality in the hypothesis,
\begin{align*}
|\mathcal{U}_Q\mathcal{W}'_{Q,q}\vec{e}|&\simeq \left(\dashint_{Q}|W^{\frac{1}{q}}(x)\mathcal{W}'_{Q,q}\vec{e}|^{s q}\right)^{\frac{1}{s q}}\\
&\lesssim \left(\dashint_{Q}|W^{\frac{1}{q}}(x)\mathcal{W}'_{Q,q}\vec{e}|^{q}\right)^{\frac{1}{q}}\simeq|\mathcal{W}_{Q,q}\mathcal{W}'_{Q,q}\vec{e}|
\end{align*}
and we are done.
\end{proof}
The following Lemma can be derived from the arguments given for the proof of Lemma 2 in \cite{IPRR}. We will provide a proof here for reader's convenience. 

\begin{lem}\label{lem:Bownik}
Let $A,B$ be self-adjoint positive definite matrices and let $0<\alpha<1$.
Then
\[|A^{\alpha}B^{\alpha}|_{op}\lesssim|AB|_{op}^\alpha\]
\end{lem}
\begin{proof} 
Let $e_{j}$ be an orthonormal basis of
eigenvalues $\lambda_{j}$ of $B$, then by the classical Hölder-McCarthy
inequality (see \cite{B} Lemma 2.1)
\begin{align*}
|A^{\alpha}B^{\alpha}|_{op} & \lesssim\sum_{j=1}^{n}|A^{\alpha}B^{\alpha}\lambda_{j}|=\sum_{j=1}^{n}\lambda_{j}^{\alpha}|A^{\alpha}e_{j}|\leq\sum_{j=1}^{n}\lambda_{j}^{\alpha}|Ae_{j}|^{\alpha}\\
&= \sum_{j=1}^{n}|A\lambda_{j}e_{j}|^{\alpha}
=\sum_{j=1}^{n}|AB^e_{j}|^{\alpha}\lesssim|AB|_{op}^{\alpha}
\end{align*}
\end{proof}
We end this section with two results that will help us to handle certain parameters in order to obtain the quantitative estimate we aim for. 
\begin{lem}\label{lem:param1}
Let $\rho>1$ and $\beta>1$, then we have that 
\[
\left(\frac{\rho'}{(\rho\beta)'}\right)^{'}\leq \rho\beta'
\]
and also that
\[\frac{1}{(\rho\beta)'}=\frac{1}{\beta'}+\frac{1}{\rho'\beta}\]
Furthermore, if $\gamma>1$ and $\beta=1+\frac{1}{\tau\kappa}$, with $\tau>2$ and $\kappa\geq1$ then
\[
\left[\left(\frac{\rho'}{(\rho\beta)'}\right)^{'}\right]^{\frac{1}{(\gamma\beta)'}}\lesssim\kappa^{\frac{1}{\gamma'}}
\]
\end{lem}
\begin{proof}
We argue as follows
\begin{align*}
\left(\frac{\rho'}{(\rho\beta)'}\right)^{'} & =\frac{\frac{\rho'}{(\rho\beta)'}}{\frac{\rho'}{(\rho\beta)'}-1}=\frac{\rho'}{\rho'-(\rho\beta)'}=\frac{\rho'}{\rho'-\frac{\rho\beta}{\rho\beta-1}}=\frac{\rho'(\rho\beta-1)}{\rho'(\rho\beta-1)-\rho\beta}=\frac{\rho'(\rho\beta-1)}{\rho'\rho\beta-\rho'-\rho\beta}\\
 & =\frac{\rho'(\rho\beta-1)}{(\rho'-1)\rho\beta-\rho'}=\frac{\rho'(\rho\beta-1)}{(\rho'-1)\frac{\rho'}{\rho'-1}\beta-\rho'}=\frac{\rho'(\rho\beta-1)}{\rho'\beta-\rho'}=\frac{\rho'(\rho\beta-1)}{\rho'(\beta-1)}=\frac{\rho\beta-1}{\beta-1}\\
 & \leq\frac{\rho\beta}{\beta-1}=\rho\beta'
\end{align*}
For the second identity first note that
\[
(\rho\beta)'=\frac{\rho\beta}{\rho\beta-1}=\frac{\frac{\rho'}{\rho'-1}\beta}{\frac{\rho'}{\rho'-1}\beta-1}=\frac{\rho'\beta}{\rho'\beta-(\rho'-1)}
\]
and taking this into account
\begin{align*}
\frac{1}{(\rho\beta)'} & =\frac{\rho'\beta-(\rho'-1)}{\rho'\beta}=\frac{\rho'\beta-\rho'+1}{\rho'\beta} =\frac{\rho'(\beta-1)+1}{\rho'\beta}=\frac{1}{\beta'}+\frac{1}{\rho'\beta}.
\end{align*}
For the last estimate note that 
\[\beta'=\frac{1+\frac{1}{\tau\kappa}}{\frac{1}{\tau\kappa}}=\tau\kappa+1\]
Then, taking into account the preceding estimate
\[
\left[\left(\frac{\rho'}{(\rho\beta)'}\right)^{'}\right]^{\frac{1}{(\gamma\beta)'}}\leq\left[\rho(\tau\kappa+1)\right]^{\frac{1}{1+\tau\kappa}+\frac{1}{\gamma'\beta}}\leq2\rho\tau\kappa^\frac{1}{\kappa}\kappa^{\frac{1}{\gamma'}}
\leq2e\rho\tau\kappa^{\frac{1}{\gamma'}}
\]
and we are done.
\end{proof}

\begin{lem}\label{lem:param2}
Let $p>1$ and $s,\beta>1$ such that 
\[
p'>s(p\beta)'
\]
\[
\beta=1+\frac{1}{\left(\frac{p'+1}{2}\right)\tau\delta}
\]
and $\beta s=1+\frac{1}{\tau\delta}$. Then

\[
\left(\frac{p'}{s(p\beta)'}\right)^{'}\leq2p\tau\delta
\]
\end{lem}
\begin{proof}
First note that
\begin{align*}
\left(\frac{p'}{s(p\beta)'}\right)^{'} & =\frac{p'}{p'-s(p\beta)'}=\frac{p'}{p'-s\frac{p\beta}{p\beta-1}}=\frac{p'(p\beta-1)}{p'(p\beta-1)-sp\beta}\\
 & =\frac{p'(p\beta-1)}{p'(p\beta-1)-sp'(p-1)\beta}=\frac{p\beta-1}{(p\beta-1)-s(p-1)\beta}
\end{align*}
It is not hard to check that
\[(p\beta-1)-s(p-1)\beta =\frac{1}{(p'+1)\tau\delta}\]
%Estas son las cuentas que prueban eso
%Now we observe that
%\begin{align*}
%(p\beta-1)-s(p-1)\beta & =p\beta-1-sp\beta+s\beta\\
% & =p\beta-1-sp\beta+s\beta\\
% & =p(\beta-s\beta)+s\beta-\beta+\beta-1\\
% & =p(\beta-s\beta)-(\beta-s\beta)+\beta-1\\
% & =(p-1)(\beta-s\beta)+\beta-1\\
% & =(p-1)\left(\frac{1}{\left(\frac{p'+1}{2}\right)\tau\delta}-\frac{1}{\tau\delta}\right)+\frac{1}{\left(\frac{p'+1}{2}\right)\tau\delta}\\
% & =\frac{p-1}{\tau\delta}\left(\frac{1}{\frac{p'+1}{2}}-\frac{\frac{p'+1}{2}}{\frac{p'+1}{2}}\right)+\frac{1}{\left(\frac{p'+1}{2}\right)\tau\delta}\\
% & =\frac{p-1}{\tau\delta}\left(\frac{2}{p'+1}-\frac{p'+1}{p'+1}\right)+\frac{1}{\left(\frac{p'+1}{2}\right)\tau\delta}\\
% & =\frac{p-1}{\tau\delta}\left(\frac{2-p'-1}{p'+1}\right)+\frac{1}{\left(\frac{p'+1}{2}\right)\tau\delta}\\
% & =\frac{p-1}{\tau\delta}\left(\frac{-p'+1}{p'+1}\right)+\frac{1}{\left(\frac{p'+1}{2}\right)\tau\delta}\\
% & =-\frac{1}{\tau\delta(p'+1)}+\frac{1}{\left(\frac{p'+1}{2}\right)\tau\delta}=\frac{1}{(p'+1)\tau\delta}
%\end{align*}
and then we can end the argument as follows
\begin{align*}
\left(\frac{p'}{s(p\beta)'}\right)^{'} & =\frac{p\beta-1}{(p\beta-1)-s(p-1)\beta}=(p\beta-1)(p'+1)\tau\delta\leq2\left(p\left(1+\frac{1}{\left(\frac{p'+1}{2}\right)\tau\delta}\right)-1\right)p'\tau\delta\\
 & =2\left(p-1+\frac{p}{\left(\frac{p'+1}{2}\right)\tau\delta}\right)p'\tau\delta \leq2p\tau\delta+\frac{2pp'\tau\delta}{\left(\frac{p'+1}{2}\right)\tau\delta}\leq2p\tau\delta
\end{align*}
\end{proof}

\section{Proofs of strong type estimates}
Note that as we pointed out in Subsection \ref{sec:endpoint}, to settle 
 \[\|T(\vec{f})\|_{L^p(W)}\lesssim c_W \|\vec{f}\|_{L^p(W)}\]
is equivalent to prove
\[\||W^{\frac{1}{p}}T(W^{-\frac{1}{p}}\vec{f})|\|_{L^p(\mathbb{R}^d)}\lesssim c_W \|\vec{f}\|_{L^p(\mathbb{R}^d)}.\]
In every proof in this section we shall settle the latter.
\subsection{Proof of Theorem \ref{Thm:RoughAp}}
Taking into account Theorem \ref{Thm:Rough} and  Lemma \ref{lem:KeyAp} we have that there exists a sparse family $\mathcal{S}$ such that

\begin{align*}
 &\left|\int_{\mathbb{R}^d}\left\langle W^{\frac{1}{p}}T_{\Omega}\left(W^{-\frac{1}{p}}\vec{h}\right),\vec{g}\right\rangle dx\right|=\left|\int_{\mathbb{R}^d}\left\langle T_{\Omega}\left(W^{-\frac{1}{p}}\vec{h}\right),W^{\frac{1}{p}}\vec{g}\right\rangle dx\right|\\
 &\leq c_{n,d}\|\Omega\|_{L^\infty(\mathbb{S}^{d-1})}s' \underset{Q\in \mathcal{S}}{\sum} \langle\langle W^{-\frac{1}{p}}\vec{h}\rangle\rangle_{1,Q}\langle\langle W^{\frac{1}{p}}\vec{g}\rangle\rangle_{s,Q}|Q|\\
 & \leq\frac{1}{\eta}c_{n,d}\|\Omega\|_{L^\infty(\mathbb{S}^{d-1})}s' \sup_{Q}\left|\mathcal{U}_{Q}\mathcal{V}_{Q}\right|_{op}
 \|M_{\mathcal{V},W^{-\frac{1}{p}},1}\|_{L^{p}}\|M_{\mathcal{U},W^{\frac{1}{p}},s}\|_{L^{p'}}\|\vec{h}\|_{L^{p}}\|\vec{g}\|_{L^{p'}}
\end{align*}
where $\{\mathcal U_Q\}$ and $\{\mathcal V_Q\}$ are families of self-adjoint positive definite matrices. Hence it suffices to show that for suitable choices of $\{\mathcal{V}_Q\}$, $\{\mathcal{U}_Q\}$ and $s>1$
\begin{equation}\label{eq:BoundRoughAp}
s' \sup_{Q}\left|\mathcal{U}_{Q}\mathcal{V}_{Q}\right|_{op}
 \|M_{\mathcal{V},W^{-\frac{1}{p}},1}\|_{L^{p}}\|M_{\mathcal{U},W^{\frac{1}{p}},s}\|_{L^{p'}}\lesssim [W]_{A_p}^\frac{1}{p}[W]_{A^{sc}_{\infty,p}}^{1+\frac{1}{p'}}[W^{-\frac{p'}{p}}]_{A^{sc}_{\infty,p'}}^\frac{1}{p}
\end{equation}
Let us choose $\mathcal{V}_{Q}$ such that for every $\vec{e}$
\begin{equation}\label{eq:CondVQRough}
|\mathcal{V}_{Q}\vec{e}|\simeq\left(\dashint_{Q}|W^{\frac{-1}{p}}(x)\vec{e}|^{rp'}\right)^{\frac{1}{rp'}}
\end{equation}
where $r=1+\dfrac{1}{2^{d+11}[W^{-\frac{p'}{p}}]_{A^{sc}_{p',\infty}}}$
and $\mathcal{U}_Q$ such that for every $\vec{e}$, 
\begin{equation}\label{eq:CondWQRough}
|\mathcal{U}_{Q}\vec{e}|\simeq\left(\dashint_{Q}|W^{\frac{1}{p}}(x)\vec{e}|^{s\gamma p}\right)^{\frac{1}{s\gamma p}}
\end{equation}
where
$$\gamma =1+\dfrac{1}{\left(\frac{p'+1}{2}\right)\tau_{d}[W]_{A^{sc}_{p,\infty}}} \qquad \text{and} \qquad  s=\left(\frac{p'+1}{2}\right)\dfrac{1+\tau_{n}[W]_{A^{sc}_{p,\infty}}}{1+\left(\frac{p'+1}{2}\right)\tau_{d}[W]_{A^{sc}_{p,\infty}}}.$$
Consequently $s\gamma=1+\dfrac{1}{2^{d+11}[W]_{A^{sc}_{p,\infty}}}$ 
and $s'\lesssim p [W]_{A^{sc}_{p,\infty}}$.
For these choices an application of Lemma \ref{lem:ApFromRH} and the definition of the $A_p$ condition yields
\begin{equation}\label{eq:ApRH}
\sup_Q|\mathcal{V}_Q\mathcal{U}_Q|_{op}\lesssim [W]_{A_p}^\frac{1}{p}.
\end{equation}
Now we focus on  $\|M_{\mathcal{V},W^{-\frac{1}{p}},1}\|_{L^{p}}$.
We are going to show that
\begin{equation}\label{eq:BoundMVRough}
\|M_{\mathcal{V},W^{-\frac{1}{p}},1}\|_{L^{p}}\lesssim[W^{-\frac{p'}{p}}]_{A^{sc}_{\infty,p'}}^\frac{1}{p}.
\end{equation}
First we observe that taking into account \eqref{eq:CondVQRough}
\begin{align*}\frac{1}{|Q|}\int_{Q}|\mathcal{V}_{Q}^{-1}W^{-\frac{1}{p}}(y)\vec{h}(y)|dy &\leq \left(\frac{1}{|Q|}\int_Q|\mathcal{V}_Q^{-1} W^{-\frac{1}{p}}(x)|^{p'r}dx)\right)^\frac{1}{p'r}\left(\frac{1}{|Q|}\int_Q|\vec{h}|^{(p'r)'}dx)\right)^\frac{1}{(p'r)'}\\
&\simeq |\mathcal{V}_Q^{-1} \mathcal{V}_Q|_{op}\left(\frac{1}{|Q|}\int_Q|\vec{h}|^{(p'r)'}dx)\right)^\frac{1}{(p'r)'}=\left(\frac{1}{|Q|}\int_Q|\vec{h}|^{(p'r)'}dx)\right)^\frac{1}{(p'r)'}
\end{align*} 
Consequently
\begin{equation}\label{eq:MVpr}
\|M_{\mathcal{V},W^{-\frac{1}{p}},1}\|_{L^{p}}\lesssim \|M_{(p'r)'}\|_{L^{p}}\lesssim \left[\left(\frac{p}{(p'r)'}\right)'\right]^\frac{1}{(p'r)'}
\end{equation}
and it suffices to provide a bound for the rightmost term. A suitable application of Lemma \ref{lem:param1} 
allows us to conclude that
\[\left[\left(\frac{p}{(p'r)'}\right)'\right]^{\frac{1}{(p'r)'}}\lesssim[W^{-\frac{p'}{p}}]_{A^{sc}_{\infty,p'}}^\frac{1}{p}.\]
This shows that \eqref{eq:BoundMVRough} holds.

At this point we turn our attention to  $\|M_{\mathcal{U},W^{\frac{1}{p}},s}\|_{L^{p'}}$.
We are going to show that
\begin{equation}\label{eq:BoundMWRough}
\|M_{\mathcal{U},W^{\frac{1}{p}},s}\|_{L^{p'}}\lesssim p'[W]^\frac{1}{p'}_{A^{sc}_{\infty,p}}.
\end{equation}
Taking into account \eqref{eq:CondWQRough}
\begin{align*}\left(\frac{1}{|Q|}\int_{Q}|\mathcal{U}_{Q}^{-1}W^{\frac{1}{p}}(y)\vec{g}(y)|^sdy\right)^{\frac{1}{s}} &\leq \left(\frac{1}{|Q|}\int_Q|\mathcal{U}_Q^{-1} W^{\frac{1}{p}}(x)|^{sp\gamma}dx)\right)^\frac{1}{sp\gamma}\left(\frac{1}{|Q|}\int_Q|\vec{g}|^{s(p\gamma)'}dx)\right)^\frac{1}{s(p\gamma)'}\\
&\simeq |\mathcal{U}_Q^{-1} \mathcal{U}_Q|_{op}\left(\frac{1}{|Q|}\int_Q|\vec{g}|^{s(p\gamma)'}dx)\right)^\frac{1}{s(p\gamma)'}=\left(\frac{1}{|Q|}\int_Q|\vec{g}|^{s(p\gamma)'}dx)\right)^\frac{1}{s(p\gamma)'}
\end{align*} 
Consequently
\begin{equation}\label{eq:MVpr}
\|M_{\mathcal{U},W^{\frac{1}{p}},s}\|_{L^{p'}}\lesssim \|M_{s(p\gamma)'}\|_{L^{p'}}\lesssim \left[\left(\frac{p'}{s(p\gamma)'}\right)'\right]^\frac{1}{s(p\gamma)'}.
\end{equation}
Now we observe that by Lemma \ref{lem:param2}
\[\left(\dfrac{p'}{s(\beta  p)'}\right)'\leq 2p\tau_n[W]_{A^{sc}_{\infty,p}}\]
and also that by Lemma \ref{lem:param1} choosing $\beta=\gamma$ and $\rho=p'$
\[\frac{1}{s(p'\gamma)'}=\frac{1}{s\gamma'}+\frac{1}{sp'\gamma}\leq \frac{c_d}{[W]_{A^{sc}_{\infty,p}}}+\frac{1}{p'}.\]
Hence, combining the estimates above
\[\left(\dfrac{p'}{s(\beta p)'}\right)'^{\frac{1}{s(\beta p)'}}\lesssim p[W]_{A^{sc}_{\infty,p}}^{\frac{1}{p'}}.\]
Combining this with \eqref{eq:MVpr} yields \eqref{eq:BoundMWRough}.
Finally taking into account our choice for $s$, \eqref{eq:ApRH}, \eqref{eq:BoundMVRough} and \eqref{eq:BoundMWRough} we conclude that \eqref{eq:BoundRoughAp} holds.

For the other estimate note that since $T^*_\Omega$ is also a rough singular integral,
\begin{align*}
 &\left|\int_{\mathbb{R}^d}\left\langle W^{\frac{1}{p}}T_{\Omega}\left(W^{-\frac{1}{p}}\vec{h}\right),\vec{g}\right\rangle dx\right|=\left|\int_{\mathbb{R}^d}\left\langle W^{-\frac{1}{p}}\vec{h},T^*_{\Omega}\left(W^{\frac{1}{p}}\vec{g}\right)\right\rangle dx\right|\\
 &\leq c_{n,d}\|\Omega\|_{L^\infty(\mathbb{S}^{d-1})}s' \underset{Q\in \mathcal{S}}{\sum} \langle\langle W^{-\frac{1}{p}}\vec{h}\rangle\rangle_{s,Q}\langle\langle W^{\frac{1}{p}}\vec{g}\rangle\rangle_{1,Q}|Q|
\end{align*}
Arguing as above essentially exchanging the roles of $\vec{h}$ and $\vec{g}$ we have that
\[s' \underset{Q\in \mathcal{S}}{\sum} \langle\langle W^{-\frac{1}{p}}\vec{h}\rangle\rangle_{s,Q}\langle\langle W^{\frac{1}{p}}\vec{g}\rangle\rangle_{1,Q}|Q|\lesssim [W]_{A_p}^\frac{1}{p}[W]_{A^{sc}_{\infty,p}}^{\frac{1}{p'}}[W^{-\frac{p'}{p}}]_{A^{sc}_{\infty,p'}}^{1+\frac{1}{p}}. \]
This ends the proof.

\subsection{Proof of Theorem \ref{Thm:HormAp}}

Taking into account Theorem \ref{Thm:Horm} and  Lemma \ref{lem:KeyAp} we have that there exists a sparse family $\mathcal{S}$ such that
\begin{align*}
 &\left|\int_{\mathbb{R}^d}\left\langle W^{\frac{1}{p}}T\left(W^{-\frac{1}{p}}\vec{h}\right),\vec{g}\right\rangle dx\right|=\left|\int_{\mathbb{R}^d}\left\langle T\left(W^{-\frac{1}{p}}\vec{h}\right),W^{\frac{1}{p}}\vec{g}\right\rangle dx\right|\\
 &\leq c_{n,d,T}\underset{Q\in \mathcal{S}}{\sum} \langle\langle W^{-\frac{1}{p}}\vec{h}\rangle\rangle_{r,Q}\langle\langle W^{\frac{1}{p}}\vec{g}\rangle\rangle_{1,Q}|Q|\\
 & \leq\frac{1}{\eta}c_{n,d,T} \sup_{Q}\left|\mathcal{U}_{Q}\mathcal{V}_{Q}\right|_{op}
 \|M_{\mathcal{V},W^{-\frac{1}{p}},r}\|_{L^{p}}\|M_{\mathcal{U},W^{\frac{1}{p}},1}\|_{L^{p'}}\|\vec{h}\|_{L^{p}}\|\vec{g}\|_{L^{p'}}
\end{align*}
where $\{\mathcal U_Q\}$ and $\{\mathcal V_Q\}$ are families of self-adjoint, positive definite matrices. Hence it suffices to show that for suitable choices of $\{\mathcal{V}_Q\}$, $\{\mathcal{U}_Q\}$
\begin{equation}\label{eq:BoundHrApr}
\sup_{Q}\left|\mathcal{U}_{Q}\mathcal{V}_{Q}\right|_{op}
 \|M_{\mathcal{V},W^{-\frac{1}{p}},r}\|_{L^{p}}\|M_{\mathcal{U},W^{\frac{1}{p}},1}\|_{L^{p'}}\lesssim [W]_{A_{\frac{p}{r}}}^{\frac{1}{p}} [W^{-\frac{r}{p}(\frac{p}{r})'}]_{A^{sc}_{(\frac{p}{r})',\infty}}^{\frac{1}{p}} [W]_{A^{sc}_{\infty,\frac{p}{r}}}^{\frac{1}{p'}}
\end{equation}  
We choose $\mathcal{U}_Q=\mathcal{A}^{\frac{1}{r}}_Q$ where \[|\mathcal{A}_{Q}\vec{e}|\simeq \left(\dashint_{Q} |W^{\frac{-r}{p}}\vec{e}|^{(\frac{p}{r})'\alpha 
  }dx\right)^{\frac{1}{(\frac{p}{r})'\alpha 
  }}\]and $\alpha=1+\dfrac{1}{\tau_{d}[W^{-\frac{r}{p}(\frac{p}{r})'}]_{A^{sc}_{(\frac{p}{r})',\infty}}}$ and $\mathcal{U}_Q=\mathcal{B}^{\frac{1}{r}}_Q$ such that \[|\mathcal{B}_{Q}\vec{e}|\simeq \left(\dashint_{Q} |W^{\frac{r}{p}}\vec{e}|^{\beta\frac{p}{r} 
  }dx\right)^{\frac{r}{\beta p}}\]where $\beta =1+\dfrac{1}{\tau_{d}[W]_{A^{sc}_{\infty,\frac{p}{r}}}}$. 
  
 First we observe that by Lemma \ref{lem:Bownik}  
 \[\left|\mathcal{U}_{Q}\mathcal{V}_{Q}\right|_{op}=
 \left|\mathcal{A}^\frac{1}{r}_{Q}\mathcal{B}_{Q}^{\frac{1}{r}}\right|_{op}
 \lesssim\left|\mathcal{A}_Q\mathcal{B}_Q\right|_{op}^{\frac{1}{r}}\]
Now by Lemma \ref{lem:ApFromRH} 
%with $q=\frac{p}{r}$, $\mathcal{V}_Q=\mathcal{A}_Q$, $\mathcal{U}_Q=\mathcal{B}_Q$, $r=\alpha$ and $s=\beta$ 
we have that
 \[\left|\mathcal{A}_Q\mathcal{B}_Q\right|_{op}^{\frac{1}{r}}\lesssim\left|\mathcal{W}_{Q,p/r}\mathcal{W}'_{Q,p/r}\right|^{\frac{1}{r}}_{op}.\]
Consequently
 \begin{equation}\label{AprRH}
 \sup_Q\left|\mathcal{U}_{Q}\mathcal{V}_{Q}\right|_{op}\leq [W]_{A_{\frac{p}{r}}}^{\frac{1}{p}}
 \end{equation}
 Now we show that 
 \begin{equation}\label{eq:MVW-Hr} 
 \|M_{\mathcal{V},W^{-\frac{1}{p}},r}\|_{L^{p}}\lesssim[W^{-\frac{r}{p}(\frac{p}{r})'}]_{A^{sc}_{(\frac{p}{r})',\infty}}^{\frac{1}{p}}
 \end{equation}
 First we observe that taking into account Lemma \ref{lem:Bownik}   and Reverse H\"older inequality  
 \begin{align*}&\left(\frac{1}{|Q|}\int_{Q}|\mathcal{A}_{Q}^{-\frac{1}{r}}W^{-\frac{1}{p}}(y)\vec{h}(y)|^{r}dy\right)^{\frac{1}{r}}\\
  &\leq \left(\frac{1}{|Q|}\int_Q|\mathcal{A}_Q^{-\frac{1}{r}} W^{-\frac{1}{p}}(x)|^{r(\frac{p}{r})'\alpha}dx)\right)^\frac{1}{r(\frac{p}{r})'\alpha}\left(\frac{1}{|Q|}\int_Q|\vec{h}|^{r((\frac{p}{r})'\alpha)'}dx)\right)^\frac{1}{r((\frac{p}{r})'\alpha)'}\\
  & =\left(\frac{1}{|Q|}\int_Q|\mathcal{A}_Q^{-\frac{1}{r}} W^{-\frac{r}{p}\frac{1}{r}}(x)|^{r(\frac{p}{r})'\alpha}dx)\right)^\frac{1}{r(\frac{p}{r})'\alpha}\left(\frac{1}{|Q|}\int_Q|\vec{h}|^{r((\frac{p}{r})'\alpha)'}dx)\right)^\frac{1}{r((\frac{p}{r})'\alpha)'}\\
 & \lesssim\left(\frac{1}{|Q|}\int_Q|\mathcal{A}_Q^{-1} W^{-\frac{r}{p}}(x)|^{(\frac{p}{r})'\alpha}dx)\right)^\frac{1}{r(\frac{p}{r})'\alpha}\left(\frac{1}{|Q|}\int_Q|\vec{h}|^{r((\frac{p}{r})'\alpha)'}dx)\right)^\frac{1}{r((\frac{p}{r})'\alpha)'}\\
&\simeq |\mathcal{A}_Q^{-1} \mathcal{A}_Q|_{op}\left(\frac{1}{|Q|}\int_Q|\vec{h}|^{r((\frac{p}{r})'\alpha)'}dx)\right)^\frac{1}{r((\frac{p}{r})'\alpha)'}=\left(\frac{1}{|Q|}\int_Q|\vec{h}|^{r((\frac{p}{r})'\alpha)'}dx)\right)^\frac{1}{r((\frac{p}{r})'\alpha)'}
\end{align*} 
from which readily follows that
\begin{equation*}
\|M_{\mathcal{V},W^{-\frac{r}{p}},r}\vec{h}\|_{L^{p}}\lesssim \|M_{r((\frac{p}{r})'\alpha)'}(\vec{h})\|_{L^{p}}\lesssim \left[\left(\frac{p}{r((\frac{p}{r})'\alpha)'}\right)'\right]^\frac{1}{r((\frac{p}{r})'\alpha)'}\|\vec{h}\|_{L^{p}}.
\end{equation*}
Now we observe that by Lemma \ref{lem:param1} we can conclude that
\[\left[\left(\dfrac{p}{r(\alpha (\frac{p}{r})')'}\right)'\right]^{\frac{1}{r(\alpha (\frac{p}{r})')'}}\lesssim[W^{\frac{r}{p}(\frac{p}{r})'}]_{A^{sc}_{(\frac{p}{r})',\infty}}^{\frac{1}{p}} \]
and hence \eqref{eq:MVW-Hr} holds.

It remains to show that 
\begin{equation}\label{eq:MVW-Hr2} 
\|M_{\mathcal{U},W^{\frac{1}{p}},1}\|_{L^{p'}}\lesssim [W]_{A^{sc}_{\infty,\frac{p}{r}}}^{\frac{1}{p'}}
\end{equation}
First note that
\begin{align*}\frac{1}{|Q|}\int_{Q}|\mathcal{B}_{Q}^{-\frac{1}{r}}W^{\frac{1}{p}}(y)\vec{g}(y)|dy &\leq\left(\frac{1}{|Q|}\int_Q|\mathcal{B}_Q^{-\frac{1}{r}} W^{\frac{1}{p}}(x)|^{p\beta}dx)\right)^\frac{1}{p\beta}\left(\frac{1}{|Q|}\int_Q|\vec{g}|^{(p\beta)'}dx)\right)^\frac{1}{(p\beta)'}\\ &=\left(\frac{1}{|Q|}\int_Q|\mathcal{B}_Q^{-\frac{1}{r}} W^{\frac{r}{p}\frac{1}{r}}(x)|^{r\frac{p}{r}\beta}dx)\right)^{\frac{r}{p\beta}\frac{1}{r}}\left(\frac{1}{|Q|}\int_Q|\vec{g}|^{(p\beta)'}dx)\right)^\frac{1}{(p\beta)'}\\
&\simeq |\mathcal{B}_Q^{-1} \mathcal{B}_{Q}|_{op}^{\frac{1}{r}}\left(\frac{1}{|Q|}\int_Q|\vec{g}|^{(p\beta)'}dx)\right)^\frac{1}{(p\beta)'}=\left(\frac{1}{|Q|}\int_Q|\vec{g}|^{(p\beta)'}dx)\right)^\frac{1}{(p\beta)'}
\end{align*} 
Consequently
\begin{equation}\label{eq:MVpr2}
\|M_{\mathcal{U},W^{\frac{1}{p}},1}\|_{L^{p'}}\lesssim \|M_{(p\beta)'}\|_{L^{p'}}\lesssim \left[\left( \frac{p'}{(p\beta)'}\right)'\right]^\frac{1}{(p\beta)'}.
\end{equation}
By Lemma \ref{lem:param1} we have that
$$\left[\left( \frac{p'}{(p\beta)'}\right)'\right]^\frac{1}{(p\beta)'}\lesssim [W]_{A^{sc}_{\infty,\frac{p}{r}}}^{\frac{1}{p'}} $$
and \eqref{eq:MVW-Hr2} holds.

Gathering \eqref{AprRH}, \eqref{eq:MVW-Hr2} and \eqref{eq:MVW-Hr} we obtain \eqref{eq:BoundHrApr} and hence we are done.

\subsection{Proof of Theorem \ref{Thm:A1}}
In virtue of Theorem \ref{Thm:Rough} applied to $T_\Omega^*$ and Lemma \ref{lem:KeyAp} we may start arguing as follows.
\begin{align*}
 &\left|\int_{\mathbb{R}^d}\left\langle W^{\frac{1}{p}}T_{\Omega}\left(W^{-\frac{1}{p}}\vec{h}\right),\vec{g}\right\rangle dx\right|=\left|\int_{\mathbb{R}^d}\left\langle T_{\Omega}\left(W^{-\frac{1}{p}}\vec{h}\right),W^{\frac{1}{p}}\vec{g}\right\rangle dx\right|\\
 &=\left|\int_{\mathbb{R}^d}\left\langle W^{-\frac{1}{p}}\vec{h},T^{*}_{\Omega}\left(W^{\frac{1}{p}}\vec{g}\right)\right\rangle dx\right|\\
 &\leq c_{n,d}\|\Omega\|_{L^\infty(\mathbb{S}^{d-1})}s' \underset{Q\in \mathcal{S}}{\sum} \langle\langle W^{-\frac{1}{p}}\vec{h}\rangle\rangle_{s,Q}\langle\langle W^{\frac{1}{p}}\vec{g}\rangle\rangle_{1,Q}|Q|\\
 & \leq\frac{1}{\eta}c_{n,d}\|\Omega\|_{L^\infty(\mathbb{S}^{d-1})}s' \sup_{Q}\left|\mathcal{U}_{Q}\mathcal{V}_{Q}\right|_{op}
 \|M_{\mathcal{V},W^{-\frac{1}{p}},s}\|_{L^{p}}\|M_{\mathcal{U},W^{\frac{1}{p}},1}\|_{L^{p'}}\|\vec{h}\|_{L^{p}}\|\vec{g}\|_{L^{p'}}
\end{align*}
Hence it suffices to bound the latter.

We take $\mathcal{V}_Q=\mathcal{A}^{-\frac{1}{p}}_Q$ where\[|\mathcal{A}_Q\vec{e}|\approx \dashint_{Q} |W\vec{e}|dx\] and $\mathcal{U}_Q=\mathcal{V}_Q^{-1}$.

For those choices we have that $\sup_{Q}\left|\mathcal{U}_{Q}\mathcal{V}_{Q}\right|_{op}=1$ 
First we show that 
\begin{equation} \label{eq:MRoughhA1}
\|M_{\mathcal{V},W^{-\frac{1}{p}},s}\|_{L^{p}}\lesssim [W]_{A_{1}}^{\frac{1}{p}}.
\end{equation}
 For that purpose we observe that taking into account the definition of $A_1$ weight and Lemma \ref{lem:Bownik}  
\begin{align*}
&\left(\dashint_{Q}|\mathcal{V}^{-1}_{Q}W^{-\frac{1}{p}}(x)\vec{h}(x)|^{s}dx\right)^{\frac{1}{s}}\\ &=\left(\dashint_{Q}|\mathcal{A}_{Q}^\frac{1}{p}W^{-\frac{1}{p}}(x)\vec{h}(x)|^{s}dx\right)^{\frac{1}{s}}
\leq\left(\dashint_{Q}|\mathcal{A}_{Q}^{\frac{1}{p}}W^{-\frac{1}{p}}(x)|_{op}^{s}|\vec{h}(x)|^{s}dx\right)^{\frac{1}{s}}\\ 
&\lesssim \left(\dashint_{Q}|\mathcal{V}_{Q}{W}^{-1}(x)|^{\frac{s}{p}}|\vec{h}(x)|^{s}dx\right)^{\frac{1}{s}}\simeq\left(\dashint_{Q}\left(\dashint_{Q}|{W}(y)W^{-1}(x)|dy\right)^{\frac{s}{p}}|\vec{h}(x)|^{s}dx\right)^{\frac{1}{s}}\\
 &\leq [W]_{A_{1}}^{\frac{1}{p}}\left(\dashint_{Q}|\vec{h}(x)|^{s}dx\right)^{\frac{1}{s}}\leq [W]_{A_{1}}^{\frac{1}{p}} M_ {s}\vec{h}(x)
 \end{align*}
 
Consequently
\begin{equation*} 
\|M_{\mathcal{V},W^{-\frac{1}{p}},s}\vec{h}\|_{L^{p}}\lesssim [W]_{A_{1}}^{\frac{1}{p}}\|M_{s}\vec{h}\|_{L^{p}}\lesssim[W]_{A_{1}}^{\frac{1}{p}}\left(\frac{p}{s}\right)'^{\frac{1}{s}}\|\vec{h}\|_{L_{p}}
\end{equation*}
and choosing for instance $s=\frac{p+1}{2}<p$ \eqref{eq:MRoughhA1} holds.

To end the proof it suffices to show that 
\begin{equation} \label{eq:MRoughhA1-2}\|M_{\mathcal{U},W^{\frac{1}{p}},1}\|_{L^{p'}}\lesssim[W]_{A^{sc}_{1,\infty}}^{\frac{1}{p'}}.
\end{equation}
Let us call $\beta=1+\dfrac{1}{2^{d+11}[W]_{A^{sc}_{1,\infty}}}$. Then we have that, taking into account Reverse H\"older inequality, and Lemma \ref{lem:Bownik}
\begin{align*}
&\frac{1}{|Q|}\int_Q|\mathcal{U}_{Q}^{-1}W^{\frac{1}{p}}(y)\vec{g}(y)|dy =\frac{1}{|Q|}\int_Q|\mathcal{A}_{Q}^{-\frac{1}{p}}W^{\frac{1}{p}}(y)\vec{g}(y)|\\
&\leq \left(\frac{1}{|Q|}\int_Q|\mathcal{A}_{Q}^{-\frac{1}{p}} W^{\frac{1}{p}}(x)|^{p\beta}dx)\right)^\frac{1}{p\beta}\left(\frac{1}{|Q|}\int_Q|\vec{g}|^{(p\beta)'}dx)\right)^\frac{1}{(p\beta)'}\\
&\lesssim \left(\frac{1}{|Q|}\int_Q|\mathcal{A}_{Q}^{-\frac{1}{p}} W^{\frac{1}{p}}(x)|_{op}^{p}dx)\right)^\frac{1}{p}\left(\frac{1}{|Q|}\int_Q|\vec{g}|^{(p\beta)'}dx)\right)^\frac{1}{(p\beta)'}\\
&\lesssim \left(\frac{1}{|Q|}\int_Q|\mathcal{A}_{Q}^{-1} W(x)|_{op}dx)\right)^\frac{1}{p}\left(\frac{1}{|Q|}\int_Q|\vec{g}|^{(p\beta)'}dx)\right)^\frac{1}{(p\beta)'}\\
&\simeq |\mathcal{A}_Q^{-1} \mathcal{A}_Q|_{op}\left(\frac{1}{|Q|}\int_Q|\vec{g}|^{(p\beta)'}dx)\right)^\frac{1}{(p\beta)'}=\left(\frac{1}{|Q|}\int_Q|\vec{g}|^{(p\beta)'}dx\right)^\frac{1}{(p\beta)'}.
\end{align*} 
Consequently
\begin{equation}
\|M_{\mathcal{U},W^{\frac{1}{p}},1}\|_{L^{p'}}\lesssim \|M_{(p\beta)'}\|_{L^{p'}}\lesssim \left[\left(\frac{p'}{(p\beta)'}\right)'\right]^\frac{1}{(p\beta)'}.
\end{equation}
Now by Lemma \ref{lem:param1} we have that
\[\left[\left(\frac{p'}{(p\beta)'}\right)'\right]^\frac{1}{(p\beta)'}\lesssim [W]_{A^{sc}_{1,\infty}}^{\frac{1}{p'}}\]
and hence \eqref{eq:MRoughhA1-2} follows.

The proof of \eqref{eq:HormA1} is exactly the same we have just presented replacing the choice we made for $s$ by $r$, that satisfies $r<p$.

\subsection{Proof of Theorem \ref{Thm:Aq}}
Again by Theorem \ref{Thm:Rough} applied to $T_\Omega^*$ and  Lemma \ref{lem:KeyAp} we may start arguing as follows.
\begin{align*}
 &\left|\int_{\mathbb{R}^d}\left\langle W^{\frac{1}{p}}T_{\Omega}\left(W^{-\frac{1}{p}}\vec{h}\right),\vec{g}\right\rangle dx\right|=\left|\int_{\mathbb{R}^d}\left\langle T_{\Omega}\left(W^{-\frac{1}{p}}\vec{h}\right),W^{\frac{1}{p}}\vec{g}\right\rangle dx\right|\\
 &=\left|\int_{\mathbb{R}^d}\left\langle W^{-\frac{1}{p}}\vec{h},T^{*}_{\Omega}\left(W^{\frac{1}{p}}\vec{g}\right)\right\rangle dx\right|\\
 &\leq c_{n,d}\|\Omega\|_{L^\infty(\mathbb{S}^{d-1})}s' \underset{Q\in \mathcal{S}}{\sum} \langle\langle W^{-\frac{1}{p}}\vec{h}\rangle\rangle_{s,Q}\langle\langle W^{\frac{1}{p}}\vec{g}\rangle\rangle_{1,Q}|Q|\\
 & \leq\frac{1}{\eta}c_{n,d}\|\Omega\|_{L^\infty(\mathbb{S}^{d-1})}s' \sup_{Q}\left|\mathcal{U}_{Q}\mathcal{V}^{\frac{1}{p}}_{Q}\right|_{op}
 \|M_{\mathcal{V},W^{-\frac{1}{p}},s}\|_{L^{p}}\|M_{\mathcal{U},W^{\frac{1}{p}},1}\|_{L^{p'}}\|\vec{h}\|_{L^{p}}\|\vec{g}\|_{L^{p'}}
\end{align*}
and it will suffice to bound the latter.
We choose $\mathcal{V}_Q=\mathcal{A}_Q^\frac{q}{p}$ where
\[|\mathcal{A}_{Q}\vec{e}|\approx \left(\frac{1}{|Q|}\int |W^{-\frac{1}{q}}(z)\vec{e}|^{q'}dz\right)^{\frac{1}{q'}}\]
and $\mathcal{U}_Q=\mathcal{V}_Q^{-1}$. For those choices
\[\sup_{Q}\left|\mathcal{U}_{Q}\mathcal{V}^{\frac{1}{p}}_{Q}\right|_{op}=1\]
so it remains to provide estimates for 
\[ \|M_{\mathcal{V},W^{-\frac{1}{p}},s}\|_{L^{p}}\qquad \text{and}\qquad \|M_{\mathcal{U},W^{\frac{1}{p}},1}\|_{L^{p'}}.\]
First we show that
\begin{equation}\label{eq:MvW-1RoughAq}
\|M_{\mathcal{V},W^{-\frac{1}{p}},1}\vec{h}\|_{L^{p}}\lesssim c_{p,q}
\end{equation}
Let $s>1$ such that $0<\frac{q}{p}s<1$ for instance we may choose $s=\frac{\frac{p}{q}+1}{2}$. Then we have that, taking that choice for $s$  and Lemma \ref{lem:Bownik} into account

\begin{align*}
&\left(\frac{1}{|Q|}\int_Q|\mathcal{V}_{Q}^{-1}W^{-\frac{1}{p}}\vec{h}(x)|^{s}dx\right)^{\frac{1}{s}}=\left(\frac{1}{|Q|}\int_Q|\mathcal{A}_{Q}^{-\frac{q}{p}}W^{-\frac{1}{p}}|^{s}|\vec{h}(x)|^{s}dx\right)^{\frac{1}{s}}\\
&\leq \left(\frac{1}{|Q|}\int_Q|\mathcal{A}_{Q}^{-\frac{q}{p}}W^{-\frac{1}{p}}|^{sq'}dx\right)^{\frac{1}{sq'}}\left(\frac{1}{|Q|}\int_Q|\vec{h}(x)|^{sq}dx\right)^{\frac{1}{sq}}\\ &\leq \left(\frac{1}{|Q|}\int_Q|\mathcal{A}_{Q}^{-1}W^{-\frac{1}{q}}|^{\frac{q}{p}sq'}dx\right)^{\frac{1}{sq'}}\left(\frac{1}{|Q|}\int_Q|\vec{h}(x)|^{sq}dx\right)^{\frac{1}{sq}}\\
&\leq\left(\frac{1}{|Q|}\int_Q|\mathcal{V}_{Q}^{-1}W^{-\frac{1}{q}}|^{q'}dx\right)^{\frac{q}{pq'}}\left(\frac{1}{|Q|}\int_Q|\vec{h}(x)|^{sq}dx\right)^{\frac{1}{sq}}\\ &\leq|\mathcal{A}_{Q}^{-1}\mathcal{A}_{Q}|^{\frac{p}{q}} \left(\frac{1}{|Q|}\int_Q|\vec{h}(x)|^{sq}dx\right)^{\frac{1}{sq}}\leqslant M_ {sq}\vec{h}(x).
 \end{align*}
Consequently
\begin{equation} 
\|M_{\mathcal{V},W^{-\frac{1}{p}},1}\vec{h}\|_{L^{p}}\lesssim \|M_{sq}\vec{h}\|_{L^{p}}\lesssim\left(\frac{p}{sq}\right)'^{\frac{1}{sq}}\|\vec{h}\|_{L^{p}}
\end{equation}
and \eqref{eq:MvW-1RoughAq} holds.

Now we turn our attention to the remaining term. We show that
\begin{equation}\label{eq:MvW-1RoughAq2}
\|M_{\mathcal{U},W^{\frac{1}{p}},1}\|_{L^{p'}}\lesssim [W]_{A_{q}}^{\frac{1}{p}}[W]_{A^{sc}_{\infty,q}}^{\frac{1}{p'}}
\end{equation}
Let $\beta=1+\dfrac{1}{2^{d+11}[W]_{A^{sc}_{\infty,q}}}$. Then we have that by reverse H\"older inequality and Lemma \ref{lem:Bownik},
\begin{align*}
\frac{1}{|Q|}\int_Q|\mathcal{U}_{Q}^{-1}W^{\frac{1}{p}}(y)\vec{g}(y)|dy &\leq \frac{1}{|Q|}\int_Q|\mathcal{A}_{Q}^{\frac{q}{p}}W^{{\frac{1}{q}\frac{q}{p}}}(y)\|\vec{g}(y)|dy\\
&\leq \left(\frac{1}{|Q|}\int_Q|\mathcal{A}_Q W^{\frac{1}{q}}(x)|^{\frac{q}{p}p\beta}dx)\right)^\frac{1}{p\beta}\left(\frac{1}{|Q|}\int_Q|\vec{g}|^{(p\beta)'}dx)\right)^\frac{1}{(p\beta)'}\\
&\leq \left(\frac{1}{|Q|}\int_Q|\mathcal{A}_Q W^{\frac{1}{q}}(x)|^{q\beta} dx)\right)^{\frac{1}{q\beta}\frac{q}{p}}\left(\frac{1}{|Q|}\int_Q|\vec{g}|^{(p\beta)'}dx\right)^\frac{1}{(p\beta)'}\\
&\lesssim\left(\frac{1}{|Q|}\int_Q|\mathcal{A}_QW^{\frac{1}{q}}(x)|^{q}dx\right)^{\frac{1}{p}}\left(\frac{1}{|Q|}\int_Q|\vec{g}|^{(p\beta)'}dx\right)^\frac{1}{(p\beta)'}\\
&\lesssim [W]_{A_{q}}^{\frac{1}{p}}\left(\frac{1}{|Q|}\int_Q|\vec{g}|^{(p\beta)'}dx\right)^{\frac{1}{(p\beta)'}}.
\end{align*} 
Hence,
\begin{equation}
\|M_{\mathcal{U},W^{\frac{1}{p}},1}\|_{L^{p'}}\lesssim [W]_{A_{q}}^{\frac{1}{p}} \|M_{(p\beta)'}\|_{L^{p'}}\lesssim [W]_{A_{q}}^{\frac{1}{p}}\left[\left(\frac{p'}{(p\beta)'}\right)'\right]^\frac{1}{(p\beta)'}.
\end{equation}
By Lemma \ref{lem:param1} it is not hard to conclude, as we did earlier in this section, that
\[\left[\left(\frac{p'}{(p\beta)'}\right)'\right]^\frac{1}{(p\beta)'}\lesssim [W]_{A^{sc}_{\infty,q}}^{\frac{1}{p'}}.\]
Gathering the estimates above, \eqref{eq:MvW-1RoughAq2} holds and we are done.

To settle \eqref{eq:HormAq} the proof is essentially the same we have just provided for \eqref{eq:RoughAq} just replacing the choice we made for $s$ by $r$ and taking into account that by hypothesis $0<\frac{q}{p}r<1$.

\section{Proofs of Coifman-Fefferman estimates}
\subsection{Estimates assuming $A_\infty$ conditions}

Let us settle each case. Let us deal with Calderón-Zygmund operators
first. By the sparse domination result in \cite{NPTV} we have that arguing
by duality,
\begin{align*}
 & \left|\int_{\mathbb{R}^{n}}\langle W^{\frac{1}{p}}(x)T(W^{-\frac{1}{p}}\vec{f})(x),\vec{g}(x)\rangle dx\right|\\
 & \lesssim\sum_{Q\in\mathcal{S}}\left(\frac{1}{|Q|}\int_{Q}|\mathcal{W}{}_{p,Q}W^{-\frac{1}{p}}\vec{f}|\right)\left(\frac{1}{|Q|}\int_{Q}|\mathcal{W}_{p,Q}^{-1}W^{\frac{1}{p}}\vec{g}|\right)|Q|\\
 & \leq\left(\sum_{Q\in\mathcal{S}}\left(\frac{1}{|Q|}\int_{Q}|\mathcal{W}{}_{p,Q}W^{-\frac{1}{p}}\vec{f}|\right)^{p}|Q|\right)^{\frac{1}{p}}\left(\sum_{Q\in\mathcal{S}}\left(\frac{1}{|Q|}\int_{Q}|\mathcal{W}_{p,Q}^{-1}W^{\frac{1}{p}}\vec{g}|\right)^{p'}|Q|\right)^{\frac{1}{p'}}\\
 & \lesssim\left\Vert \sup_{Q}\frac{1}{|Q|}\int_{Q}|\mathcal{W}{}_{p,Q}W^{-\frac{1}{p}}\vec{f}|\right\Vert _{L^{p}}\left(\sum_{Q\in\mathcal{S}}\left(\frac{1}{|Q|}\int_{Q}|\mathcal{W}_{p,Q}^{-1}W^{\frac{1}{p}}\vec{g}|\right)^{p'}|Q|\right)^{\frac{1}{p'}}.
\end{align*}
Now we observe that if $W\in A_{\infty,p}^{sc}$ we have that choosing
$r=1+\frac{1}{2^{d+11}[W]_{A_{\infty,p}^{sc}}}$
\begin{align*}
\frac{1}{|Q|}\int_{Q}|\mathcal{W}_{p,Q}^{-1}W^{\frac{1}{p}}\vec{g}| & \leq\left(\frac{1}{|Q|}\int_{Q}|\mathcal{W}_{p,Q}^{-1}W^{\frac{1}{p}}|_{op}^{rp}\right)^{\frac{1}{rp}}\left(\frac{1}{|Q|}\int_{Q}|\vec{g}|^{(rp)'}\right)^{\frac{1}{(rp)'}}\\
 & \lesssim\left(\frac{1}{|Q|}\int_{Q}|\mathcal{W}_{p,Q}^{-1}W^{\frac{1}{p}}|_{op}^{p}\right)^{\frac{1}{p}}\left(\frac{1}{|Q|}\int_{Q}|\vec{g}|^{(rp)'}\right)^{\frac{1}{(rp)'}}\\
 & \lesssim\left(\frac{1}{|Q|}\int_{Q}|\vec{g}|^{(rp)'}\right)^{\frac{1}{(rp)'}}
\end{align*}
and hence
\[
\left(\sum_{Q\in\mathcal{S}}\langle\mathcal{W}{}_{Q}^{-1}W^{\frac{1}{p}}g\rangle_{Q}^{p'}|Q|\right)^{\frac{1}{p'}}\lesssim\left\Vert M_{(rp)'}(|\vec{g}|)\right\Vert _{L^{p'}}\lesssim\left(\frac{p'}{(rp)'}\right)^{\frac{1}{(rp)'}}\|\vec{g}\|_{L^{p'}}.
\]
Since by Lemma \ref{lem:param1}
\[
\left(\frac{p'}{(rp)'}\right)^{\frac{1}{(rp)'}}\lesssim[W]_{A_{\infty,p}^{sc}}^{\frac{1}{p}}
\]
we are done.

For $T_{\Omega}$, by Theorem \ref{Thm:Rough} we have that
\begin{align*}
 & \left|\int_{\mathbb{R}^{n}}\langle W^{\frac{1}{p}}(x)T_{\Omega}(W^{-\frac{1}{p}}\vec{f})(x),\vec{g}(x)\rangle dx\right|\\
 & \lesssim\|\Omega\|_{L^{\infty}(\mathbb{S}^{d-1})}s'\sum_{Q\in\mathcal{S}}\left(\frac{1}{|Q|}\int_{Q}|\mathcal{W}{}_{p,Q}W^{-\frac{1}{p}}\vec{f}|\right)\left(\frac{1}{|Q|}\int_{Q}|\mathcal{W}_{p,Q}^{-1}W^{\frac{1}{p}}\vec{g}|^{s}\right)^{\frac{1}{s}}|Q|\\
 & \leq\left(\sum_{Q\in\mathcal{S}}\left(\frac{1}{|Q|}\int_{Q}|\mathcal{W}{}_{p,Q}W^{-\frac{1}{p}}\vec{f}|\right)^{p}|Q|\right)^{\frac{1}{p}}\left(\sum_{Q\in\mathcal{S}}\left(\frac{1}{|Q|}\int_{Q}|\mathcal{W}_{p,Q}^{-1}W^{\frac{1}{p}}\vec{g}|^{s}\right)^{\frac{p'}{s}}|Q|\right)^{\frac{1}{p'}}\\
 & \lesssim\|\Omega\|_{L^{\infty}(\mathbb{S}^{d-1})}s'\left\Vert \sup_{Q}\frac{1}{|Q|}\int_{Q}|\mathcal{W}{}_{p,Q}W^{-\frac{1}{p}}\vec{f}|\right\Vert _{L^{p}}\left(\sum_{Q\in\mathcal{S}}\left(\frac{1}{|Q|}\int_{Q}|\mathcal{W}_{p,Q}^{-1}W^{\frac{1}{p}}\vec{g}|^{s}\right)^{\frac{p'}{s}}|Q|\right)^{\frac{1}{p'}}
\end{align*}
Choosing $s=\left(\frac{p'+1}{2}\right)\frac{1+\tau_{d}[W]_{A_{p,\infty}^{sc}}}{1+\left(\frac{p'+1}{2}\right)\tau_{d}[W]_{A_{p,\infty}^{sc}}}$
and $r=1+\frac{1}{\left(\frac{p'+1}{2}\right)\tau_{d}[W]_{A_{p,\infty}^{sc}}}$
\begin{align*}
\left(\frac{1}{|Q|}\int_{Q}|\mathcal{W}_{p,Q}^{-1}W^{\frac{1}{p}}\vec{g}|^{s}\right)^{\frac{1}{s}} & \leq\left(\frac{1}{|Q|}\int_{Q}|\mathcal{W}_{p,Q}^{-1}W^{\frac{1}{p}}|_{op}^{sp}\right)^{\frac{1}{srp}}\left(\frac{1}{|Q|}\int_{Q}|\vec{g}|^{s(rp)'}\right)^{\frac{1}{s(rp)'}}\\
 & \lesssim\left(\frac{1}{|Q|}\int_{Q}|\mathcal{W}_{p,Q}^{-1}W^{\frac{1}{p}}|_{op}^{p}\right)^{\frac{1}{p}}\left(\frac{1}{|Q|}\int_{Q}|\vec{g}|^{s(rp)'}\right)^{\frac{1}{s(rp)'}}\\
 & \lesssim\left(\frac{1}{|Q|}\int_{Q}|\vec{g}|^{s(rp)'}\right)^{\frac{1}{s(rp)'}}
\end{align*}
Hence arguing as we did to settle \eqref{eq:MVpr}
\[
\left(\sum_{Q\in\mathcal{S}}\left(\frac{1}{|Q|}\int_{Q}|\mathcal{W}_{p,Q}^{-1}W^{\frac{1}{p}}\vec{g}|\right)^{p'}|Q|\right)^{\frac{1}{p'}}\lesssim\|M_{s(rp)'}(|\vec{g}|)\|_{L^{p'}}\lesssim[W]_{A_{p,\infty}^{sc}}^{\frac{1}{p}}\|\vec{g}\|_{L^{p'}}
\]
and we are done.

Finally if $T$ is a $L^{r'}$-Hörmander operator
\begin{align*}
 & \left|\int_{\mathbb{R}^{n}}\langle W^{\frac{1}{p}}(x)T(W^{-\frac{1}{p}}\vec{f})(x),\vec{g}(x)\rangle dx\right|\\
 & \lesssim c_{n,d,T}\sum_{Q\in\mathcal{S}}\left(\frac{1}{|Q|}\int_{Q}|\mathcal{W}{}_{\frac{p}{r},Q}^{\frac{1}{r}}W^{-\frac{1}{p}}\vec{f}|^{r}\right)^{\frac{1}{r}}\left(\frac{1}{|Q|}\int_{Q}|\mathcal{W}_{\frac{p}{r},Q}^{-\frac{1}{r}}W^{\frac{1}{p}}\vec{g}|\right)|Q|\\
 & \leq c_{n,d,T}\left(\sum_{Q\in\mathcal{S}}\left(\frac{1}{|Q|}\int_{Q}|\mathcal{W}{}_{p,Q}W^{-\frac{1}{p}}\vec{f}|^{r}\right)^{\frac{p}{r}}|Q|\right)^{\frac{1}{p}}\left(\sum_{Q\in\mathcal{S}}\left(\frac{1}{|Q|}\int_{Q}|\mathcal{W}_{p,Q}^{-1}W^{\frac{1}{p}}\vec{g}|\right)^{p'}|Q|\right)^{\frac{1}{p'}}\\
 & \lesssim c_{n,d,T}\left\Vert \sup_{Q}\left(\frac{1}{|Q|}\int_{Q}|\mathcal{W}{}_{\frac{p}{r},Q}^{\frac{1}{r}}W^{-\frac{1}{p}}\vec{f}|^{r}\right)^{\frac{1}{r}}\right\Vert _{L^{p}}\left(\sum_{Q\in\mathcal{S}}\left(\frac{1}{|Q|}\int_{Q}|\mathcal{W}_{\frac{p}{r},Q}^{-\frac{1}{r}}W^{\frac{1}{p}}\vec{g}|\right)^{p'}|Q|\right)^{\frac{1}{p'}}.
\end{align*}
Taking into account Lemma \ref{lem:Bownik} and Reverse Hölder inequality
\begin{align*}
\frac{1}{|Q|}\int_{Q}|\mathcal{W}_{\frac{p}{r},Q}^{-\frac{1}{r}}W^{\frac{1}{p}}\vec{g}| & \leq\left(\frac{1}{|Q|}\int_{Q}|\mathcal{W}_{\frac{p}{r},Q}^{-\frac{1}{r}}W^{\frac{1}{p}}\vec{g}|^{\alpha p}\right)^{\frac{1}{\alpha p}}\left(\frac{1}{|Q|}\int_{Q}|\vec{g}|^{(p\alpha)'}\right)^{\frac{1}{(p\alpha)'}}\\
 & \lesssim\left(\frac{1}{|Q|}\int_{Q}|\mathcal{W}_{\frac{p}{r},Q}^{-1}W^{\frac{1}{p/r}}|^{\frac{p\alpha}{r}}\right)^{\frac{1}{p\alpha}}\left(\frac{1}{|Q|}\int_{Q}|\vec{g}|^{(p\alpha)'}\right)^{\frac{1}{(p\alpha)'}}\\
 & \lesssim\left(\frac{1}{|Q|}\int_{Q}|\mathcal{W}_{\frac{p}{r},Q}^{-1}W^{\frac{1}{p/r}}|^{\frac{p}{r}}\right)^{\frac{1}{p}}\left(\frac{1}{|Q|}\int_{Q}|\vec{g}|^{(p\alpha)'}\right)^{\frac{1}{(p\alpha)'}}\\
 & \lesssim\left(\frac{1}{|Q|}\int_{Q}|\vec{g}|^{(p\alpha)'}\right)^{\frac{1}{(p\alpha)'}}.
\end{align*}
From this point arguing as we did for Calderón-Zygmund operators we
deduce that
\begin{align*}
\left(\sum_{Q\in\mathcal{S}}\left(\frac{1}{|Q|}\int_{Q}|\mathcal{W}_{\frac{p}{r},Q}^{-\frac{1}{r}}W^{\frac{1}{p}}\vec{g}|\right)^{p'}|Q|\right)^{\frac{1}{p'}} & \leq\left\Vert \sup_{Q}\frac{1}{|Q|}\int_{Q}|\mathcal{W}_{\frac{p}{r},Q}^{-1}W^{\frac{1}{p/r}}\vec{g}|^{\frac{1}{r}}\right\Vert _{L^{p'}}\\
 & \lesssim[W]_{A_{\infty,\frac{p}{r}}^{sc}}^{\frac{1}{p}}\|\vec{g}\|_{L^{p'}}
\end{align*}
and we are done.

\subsection{Estimates assuming $C_p$ type conditions}

Note that arguing by duality, exactly the same argument provided above
works, provided we are able to adapt in each case the term involving
$\vec{g}$. We begin with Calderón-Zygmund operators. Note that
\begin{align*}
\left(\sum_{Q\in\mathcal{S}}\left(\frac{1}{|Q|}\int_{Q}|\mathcal{W}_{p,Q}^{-1}W^{\frac{1}{p}}\vec{g}|\right)^{p'}|Q|\right)^{\frac{1}{p'}} & \simeq\left\Vert \sum_{Q\in\mathcal{S}}\left(\frac{1}{|Q|}\int_{Q}|\mathcal{W}_{p,Q}^{-1}W^{\frac{1}{p}}\vec{g}|\right)\chi_{Q}\right\Vert _{L^{p'}}.
\end{align*}
Hence, arguing by duality
\begin{align*}
 & \sum_{Q\in\mathcal{S}}\left(\frac{1}{|Q|}\int_{Q}|\mathcal{W}{}_{p,Q}^{-1}W^{\frac{1}{p}}\vec{g}|\right)\int_{Q}h\\
 & =\sum_{Q\in\mathcal{S}}\left(\frac{1}{|Q|}\int_{Q}|\mathcal{W}{}_{p,Q}^{-1}W^{\frac{1}{p}}|^{\gamma p}\right)^{\frac{1}{\gamma p}}\left(\frac{1}{|Q|}\int|\vec{g}|^{(\gamma p)'}\right)^{\frac{1}{(\gamma p)'}}\frac{1}{|Q|}\int_{Q}h|Q|\\
 & \lesssim\sum_{Q\in\mathcal{S}}\left(\frac{1}{|Q|}\int_{Q}|\mathcal{W}{}_{p,Q}^{-1}W^{\frac{1}{p}}|^{\gamma p}\right)^{\frac{1}{\gamma p}}\left(\frac{1}{|Q|}\int|\vec{g}|^{(\gamma p)'}\right)^{\frac{1}{(\gamma p)'}}\frac{1}{|Q|}\int_{Q}h|Q|\\
 & \lesssim\sum_{Q\in\mathcal{S}}\left(\frac{1}{|Q|}\int_{\mathbb{R}^{n}}M(\chi_{Q})^{q}\right)^{\frac{1}{p}}\left(\frac{1}{|Q|}\int|\vec{g}|^{(\gamma p)'}\right)^{\frac{1}{(\gamma p)'}}\frac{1}{|Q|}\int_{Q}h|Q|\\
 & \lesssim\left(\sum_{Q\in\mathcal{S}}\int_{\mathbb{R}^{n}}M(\chi_{Q})^{q}\left(\frac{1}{|Q|}\int_{Q}h\right)^{p}\right)^{\frac{1}{p}}\left(\sum_{Q\in\mathcal{S}}\left(\frac{1}{|Q|}\int|\vec{g}|^{(rp)'}\right)^{\frac{p'}{(rp)'}}|E_{Q}|\right)^{\frac{1}{p'}}\\
 & \lesssim\|h\|_{L^{p}}\|M_{(rp)'}(|\vec{g}|)\|_{L^{p'}}\lesssim\|h\|_{L^{p}}\||\vec{g}|\|_{L^{p'}}
\end{align*}
where the bound for $h$ in the last inequality follows from Lemma
\cite[Corollary 3.7]{CLRT} with $w=1$.

Analogously, in the case of rough singular integrals, choosing $s>1$
and $\alpha>1$ such that $1<\alpha s<\gamma$ and $1<s(\alpha p)'<p'$
we have that
\begin{align*}
 & \sum_{Q\in\mathcal{S}}\left(\frac{1}{|Q|}\int_{Q}|\mathcal{W}{}_{p,Q}^{-1}W^{\frac{1}{p}}\vec{g}|^{s}\right)^{\frac{1}{s}}\int_{Q}h\\
 & =\sum_{Q\in\mathcal{S}}\left(\frac{1}{|Q|}\int_{Q}|\mathcal{W}{}_{p,Q}^{-1}W^{\frac{1}{p}}|^{s\alpha p}\right)^{\frac{1}{s\alpha p}}\left(\frac{1}{|Q|}\int|\vec{g}|^{s(\alpha p)'}\right)^{\frac{1}{\alpha(\alpha p)'}}\frac{1}{|Q|}\int_{Q}h|Q|\\
 & \lesssim\sum_{Q\in\mathcal{S}}\left(\frac{1}{|Q|}\int_{Q}|\mathcal{W}{}_{p,Q}^{-1}W^{\frac{1}{p}}|^{\gamma p}\right)^{\frac{1}{\gamma p}}\left(\frac{1}{|Q|}\int|\vec{g}|^{s(\alpha p)'}\right)^{\frac{1}{\alpha(\alpha p)'}}\frac{1}{|Q|}\int_{Q}h|Q|\\
 & \leq\sum_{Q\in\mathcal{S}}\left(\frac{1}{|Q|}\int_{\mathbb{R}^{n}}M(\chi_{Q})^{q}\right)^{\frac{1}{p}}\left(\frac{1}{|Q|}\int|\vec{g}|^{s(\alpha p)'}\right)^{\frac{1}{\alpha(\alpha p)'}}\frac{1}{|Q|}\int_{Q}h|Q|\\
 & \lesssim\left(\sum_{Q\in\mathcal{S}}\int_{\mathbb{R}^{n}}M(\chi_{Q})^{q}\left(\frac{1}{|Q|}\int_{Q}h\right)^{p}\right)^{\frac{1}{p}}\left(\sum_{Q\in\mathcal{S}}\left(\frac{1}{|Q|}\int|\vec{g}|^{s(\alpha p)'}\right)^{\frac{p'}{\alpha(\alpha p)'}}|E_{Q}|\right)^{\frac{1}{p'}}\\
 & \lesssim\|h\|_{L^{p}}\|M_{s(\alpha p)'}(|\vec{g}|)\|_{L^{p'}}\lesssim\|h\|_{L^{p}}\||\vec{g}|\|_{L^{p'}}
\end{align*}
where, again, the bound for $h$ in the last inequality follows from
Lemma \cite[Corollary 3.7]{CLRT} with $w=1$.

In the case of $L^{r}$-Hörmander operators, we have that
\begin{align*}
 & \sum_{Q\in\mathcal{S}}\left(\frac{1}{|Q|}\int_{Q}|\mathcal{W}{}_{\frac{p}{r},Q}^{-\frac{1}{r}}W^{\frac{1}{p}}\vec{g}|\right)\int_{Q}h\\
 & =\sum_{Q\in\mathcal{S}}\left(\frac{1}{|Q|}\int_{Q}|\mathcal{W}{}_{\frac{p}{r},Q}^{-\frac{1}{r}}W^{\frac{r}{rp}}|^{\gamma p}\right)^{\frac{1}{\gamma p}}\left(\frac{1}{|Q|}\int|\vec{g}|^{(\gamma p)'}\right)^{\frac{1}{(\gamma p)'}}\frac{1}{|Q|}\int_{Q}h|Q|\\
 & \lesssim\sum_{Q\in\mathcal{S}}\left(\frac{1}{|Q|}\int_{Q}|\mathcal{W}{}_{\frac{p}{r},Q}^{-1}W^{\frac{1}{p}}|^{\frac{\gamma p}{r}}\right)^{\frac{1}{\gamma p}}\left(\frac{1}{|Q|}\int|\vec{g}|^{(\gamma p)'}\right)^{\frac{1}{(\gamma p)'}}\frac{1}{|Q|}\int_{Q}h|Q|\\
 & \leq\sum_{Q\in\mathcal{S}}\left(\frac{1}{|Q|}\int_{\mathbb{R}^{n}}M(\chi_{Q})^{q}\right)^{\frac{1}{p}}\left(\frac{1}{|Q|}\int|\vec{g}|^{(\gamma p)'}\right)^{\frac{1}{(\gamma p)'}}\frac{1}{|Q|}\int_{Q}h|Q|
\end{align*}
and the remainder of the proof is the same as in the case of Calderón-Zygmund
operators.

\section{Proofs of endpoint estimates}
\subsection{Proof of Theorem \ref{eq:RoughEndpoint}}
The argument is an adaption of the one used in \cite{CUIMPRR} for the endpoint estimate of the commutator. We reproduce the full argument here for reader's convenience.

Without loss of generality we may assume that $\lambda=1$ and $\|\vec{f}\|_{L^{1}}=\|\Omega\|_{L^{\infty}(\mathbb{S}^{d-1})}=1$.
If 
\[
G=\{|W(x)T_{\Omega}(W^{-1}f)(x)|>1\}\setminus\{M(|\vec{f}|)(x)>1\},
\]
then it will suffice to prove that 
\[
|G|\leq c_{n,d}[W]_{A_{1}}[W]_{A_{\infty,1}^{sc}}\max\left\{ \log\left([W]_{A_{1}}+e\right),[W]_{A_{\infty,1}^{sc}}\right\} +\frac{1}{2}|G|.
\]
Let $e_{i}$ the canonic basis in $\mathbb{R}^{n}$ and let us consider
$\vec{g}=\chi_{G}\sum_{i=1}^{n}e_{i}$. We then have that for $s>1$ to be chosen, by sparse domination,
\begin{align*}
|G| & =\left|\left\{ x\in G:|W(x)T_{\Omega}(W^{-1}f)(x)>1\right\} \right|\leq\int_{G}|W(x)T_{\Omega}(W^{-1}f)(x)|dx\\
 & \leq c_{n}|\int_{\mathbb{R}^{d}}\left\langle W(x)T_{\Omega}(W^{-1}f)(x),\vec{g}(x)\rangle\right| dx=c_{n}\int_{\mathbb{R}^{d}}\left|\langle T_{\Omega}(W^{-1}f)(x),W(x)\vec{g}(x)\rangle\right| dx\\
 & \le c_{n,d}\|\Omega\|_{L^{\infty}(\mathbb{S}^{d-1})}r'\sum_{Q\in\mathcal{S}}\langle\langle W^{-1}\vec{f}\rangle\rangle_{1,Q}\langle\langle W\vec{g}\rangle\rangle_{r,Q}|Q|
\end{align*}
Now we observe that choosing $r=s=1+\frac{1}{3\cdot2^{d+11}[W]_{A_{\infty,1}^{sc}}}$
we have that $r'=s'\simeq[W]_{A_{\infty,1}^{sc}}$ and that 
\[
rs\leq1+\frac{1}{2^{d+11}[W]_{A_{\infty,1}^{sc}}}.
\]
Relying upon this choice we have that arguing similarly as we did to settle Lemma \ref{lem:KeyAp}
\begin{align*}
 & r'\langle\langle W^{-1}\vec{f}\rangle\rangle_{1,Q}\langle\langle W\vec{g}\rangle\rangle_{s,Q}|Q|\\
 & =r'\langle\langle\mathcal{W}_{1,Q}W^{-1}\vec{f}\rangle\rangle_{1,Q}\langle\langle\mathcal{W}_{1,Q}^{-1}W\vec{g}\rangle\rangle_{r,Q}|Q|\\
 & \leq r'\left(\frac{1}{|Q|}\int_{Q}|\mathcal{W}_{1,Q}W^{-1}\vec{f}|\right)\left(\frac{1}{|Q|}\int_{Q}|\mathcal{W}_{1,Q}^{-1}W\vec{g}|^{r}\right)^{\frac{1}{r}}\\
 & \leq c_{n,d}[W]_{A_{1}}[W]_{A_{\infty,1}^{sc}}\left(\frac{1}{|Q|}\int_{Q}|\vec{f}|\right)\left(\frac{1}{|Q|}\int_{Q}|\mathcal{W}_{1,Q}^{-1}W|^{rs}\right)^{\frac{1}{rs}}\left(\frac{1}{|Q|}\int_{Q}|\vec{g}|^{s'r}\right)^{\frac{1}{sr'}}\\
 & \leq c_{n,d}[W]_{A_{1}}[W]_{A_{\infty,1}^{sc}}\left(\frac{1}{|Q|}\int_{Q}|\vec{f}|\right)\left(\frac{1}{|Q|}\int_{Q}|\mathcal{W}_{1,Q}^{-1}W|\right)\left(\frac{1}{|Q|}\int_{Q}|\vec{g}|^{s'r}\right)^{\frac{1}{s'r}}\\
 & \leq c_{n,d}[W]_{A_{1}}[W]_{A_{\infty,1}^{sc}}\left(\frac{1}{|Q|}\int_{Q}|\vec{f}|\right)\left(\frac{1}{|Q|}\int_{Q}g\right)^{\frac{1}{s'r}}
\end{align*}
where $g=\chi_{G}$, and consequently, 
\[
|G|\leq c_{n,d}[W]_{A_{1}}[W]_{A_{\infty,1}^{sc}}\sum_{Q\in\mathcal{S}}\left(\frac{1}{|Q|}\int_{Q}|\vec{f}|\right)\left(\frac{1}{|Q|}\int_{Q}g\right)^{\frac{1}{s'r}}|Q|
\]
We may assume that $\mathcal{S}$ is $\frac{4}{5}$-sparse. Otherwise
we may split the sparse family $\mathcal{S}$ by \cite[Lemma 6.6]{LN}
and deal just with the maximum of the resulting sparse family sums
times a constant depending only on the sparse constant.

Being that reduction done we now split the sparse family as follows.
We say that $Q\in\mathcal{S}_{k,j}$, $k,j\geq0$ if 
\[
2^{-j-1}<\frac{1}{|Q|}\int_{Q}|\vec{f}(y)|dy\leq2^{-j},\qquad2^{-k-1}<\left(\frac{1}{|Q|}\int_{Q}|g(y)|dy\right)^{\frac{1}{s'r}}\leq2^{-k}.
\]
Then we can write
\begin{align*}
|G| & \leq c_{n,d}[W]_{A_{1}}[W]_{A_{\infty,1}^{sc}}\sum_{j=0}^{\infty}\sum_{k=0}^{\infty}\sum_{Q\in\mathcal{S}_{k,j}}\left(\frac{1}{|Q|}\int_{Q}g(x)\,dx\right)^{\frac{1}{rs'}}\frac{1}{|Q|}\int_{Q}|f(y)|\,dy|Q|\\
 & :=c_{n,d}[W]_{A_{1}}[W]_{A_{\infty,1}^{sc}}\sum_{k=0}^{\infty}\sum_{j=0}^{\infty}s_{k,j}.
\end{align*}
We claim that 
\[
s_{k,j}\leq c_{n,d}[W]_{A_{1}}[W]_{A_{\infty,1}^{sc}}\min\left\{ 2\cdot2^{-k},c_{n,d}2^{-j}2^{k(rs'-1)+rs'}|G|\right\} :=\alpha_{k,j}.
\]
For the first bound we argue as follows. Let $E_{Q}=Q\setminus\bigcup_{\substack{Q'\in\mathcal{S}_{j,k}\\
Q'\subsetneq Q
}
}Q'$. Then 
\begin{align*}
\int_{Q}|\vec{f}(y)|dy & =\int_{E_{Q}}|\vec{f}(y)|dy+\int_{\bigcup_{\substack{Q'\in\mathcal{S}_{j,k}\\
Q'\subsetneq Q
}
}}|\vec{f}(y)|dy\\
 & \leq\int_{E_{Q}}|\vec{f}(y)|dy+\sum_{\substack{Q'\in\mathcal{S}_{j,k}\\
Q'\subsetneq Q
}
}\int_{Q'}|\vec{f}(y)|dy.
\end{align*}
For the second term on the right hand side, we have that since $\mathcal{S}$
is $\frac{4}{5}$-sparse and hence it is $\frac{5}{4}$-Carleson
\[
\sum_{{\substack{Q'\in\mathcal{S}_{j,k}\\
Q'\subsetneq Q
}
}}\int_{Q'}|\vec{f}(y)|dy\leq2^{-j}\sum_{{\substack{Q'\in\mathcal{S}_{j,k}\\
Q'\subsetneq Q
}
}}|Q'|\leq2^{-j-2}|Q|\leq\frac{1}{2}\int_{Q}|\vec{f}(y)|\,dy.
\]
Thus, 
\[
\int_{Q}|\vec{f}(y)|dy\leq2\int_{E_{Q}}|\vec{f}(y)|\,dy,
\]
from which readily follows that
\begin{align*}
s_{k,j}&\leq2\sum_{Q\in\mathcal{S}_{j,k}}\int_{E_{Q}}|\vec{f}(y)|\,dy\left(\frac{1}{|Q|}\int_{Q}g(x)\,dx\right)^{\frac{1}{rs'}}\\
&\leq2\cdot2^{-k}\sum_{Q\in\mathcal{S}_{j,k}}\int_{E_{Q}}|f(y)|\,dy\leq2\cdot2^{-k}\int_{\mathbb{R}^{d}}|\vec{f}(y)|\,dy=2\cdot2^{-k}.
\end{align*}

For the second estimate of $s_{k,j}$, let $\mathcal{S}_{j,k}^{*}$
denote the maximal cubes in $\mathcal{S}_{j,k}$. Then, taking into
account again that $\mathcal{S}$ is $\frac{5}{4}$-Carleson, 
\begin{equation}
\begin{split}s_{k,j} & \leq2^{-j}2^{-k}\sum_{Q\in\mathcal{S}_{j,k}}|Q|\leq2^{-j}2^{-k}\sum_{Q\in\mathcal{S}_{j,k}^{*}}\sum_{P\subseteq Q}|P|\\
 & \leq\frac{5}{4}2^{-j}2^{-k}\sum_{Q\in\mathcal{S}_{j,k}^{*}}|Q|=\frac{5}{4}2^{-j}2^{-k}\bigg|\bigcup_{Q\in\mathcal{S}_{j,k}}Q\bigg|\\
 & =\frac{5}{4}2^{-j}2^{-k}\left|\left\{ x\in\mathbb{R}^{d}\,:\,Mg(x)>2^{-rs'k-rs'}\right\} \right|\\
 & \leq c_{d}2^{-j}2^{k(rs'-1)+rs'}|G|;
\end{split}
\label{eq:G-1}
\end{equation}

Combining the estimates above, we obtain 
\[
|G|\leq c_{n,d}[W]_{A_{1}}[W]_{A_{\infty,1}^{sc}}\sum_{k=0}^{\infty}\sum_{j=0}^{\infty}\alpha_{k,j}.
\]
Fix $\gamma>0$, to be chosen later on. To complete the proof we decompose
the double sum as follows. 
\begin{align*}
\sum_{k=0}^{\infty}\sum_{j=0}^{\infty}\alpha_{k,j} & =\sum_{j\geq\left\lceil \log_{2}\left([W]_{A_{1}}[W]_{A_{\infty}}\gamma\right)\right\rceil +\left\lceil k(2s'-1)+2s'\right\rceil +k}\alpha_{k,j}\\
 & \qquad+\sum_{j<\left\lceil \log_{2}\left([W]_{A_{1}}[W]_{A_{\infty}}\gamma\right)\right\rceil +\left\lceil k(rs'-1)+rs'\right\rceil +k}\alpha_{k,j}.
\end{align*}

To estimate the first sum on the right, note that 
\begin{align*}
 & \sum_{j\geq\left\lceil \log_{2}\left([W]_{A_{1}}[W]_{A_{\infty}}\gamma\right)\right\rceil +\left\lceil k(rs'-1)+rs'\right\rceil +k}\alpha_{k,j}\\
 & \leq c_{n,d}[W]_{A_{1}}[W]_{A_{\infty,1}^{sc}}|G|\sum_{k=0}^{\infty}2^{k(rs'-1)+rs'}\sum_{j\geq\left\lceil \log_{2}\left([W]_{A_{1}}[W]_{A_{\infty,1}^{sc}}\gamma\right)\right\rceil +\left\lceil k(rs'-1)+rs'\right\rceil +k}2^{-j}\\
 & =c_{n,d}[W]_{A_{1}}[W]_{A_{\infty,1}^{sc}}|G|\sum_{k=0}^{\infty}2^{k(rs'-1)+rs'}2^{-\left\lceil \log_{2}\left([W]_{A_{1}}[W]_{A_{\infty,1}^{sc}}\gamma\right)\right\rceil -\left\lceil k(rs'-1)+rs'\right\rceil -k}\\
 & =c_{n,d}[W]_{A_{1}}[W]_{A_{\infty,1}^{sc}}|G|\sum_{k=0}^{\infty}2^{k(rs'-1)+rs'}2^{-\left\lceil \log_{2}\left([W]_{A_{1}}[W]_{A_{\infty,1}^{sc}}\gamma\right)\right\rceil -\left\lceil k(rs'-1)+rs'\right\rceil -k}\\
 & \leq\frac{c_{n,d}[W]_{A_{1}}[W]_{A_{\infty,1}^{sc}}}{[W]_{A_{1}}[W]_{A_{\infty,1}^{sc}}\gamma}|G|\sum_{k=0}^{\infty}2^{-k}\leq\frac{2c_{n,d}}{\gamma}|G|.
\end{align*}
Therefore, it suffices to let $\gamma=4c_{n,d}$.

To estimate the second sum on the right, note that
\begin{align*}
 & \sum_{j<\left\lceil \log_{2}\left([W]_{A_{1}}[W]_{A_{\infty,1}^{sc}}\gamma\right)\right\rceil +\left\lceil k(rs'-1)+rs'\right\rceil +k}\alpha_{k,j}\\
 & \leq c_{n,d}\sum_{k=0}^{\infty}\sum_{1\leq j<\left\lceil \log_{2}\left([W]_{A_{1}}[W]_{A_{\infty,1}^{sc}}\gamma\right)\right\rceil +\left\lceil k(rs'-1)+rs'\right\rceil +k}2^{-k}[W]_{A_{1}}[W]_{A_{\infty,1}^{sc}}\\
 & \leq c_{n,d}\sum_{k=0}^{\infty}\left(\log_{2}\left([W]_{A_{1}}[W]_{A_{\infty,1}^{sc}}4c_{d}\right)+krs'\right)2^{-k}[W]_{A_{1}}[W]_{A_{\infty,1}^{sc}}\\
 & \leq c_{n,d}[W]_{A_{1}}[W]_{A_{\infty,1}^{sc}}\max\left\{ \log\left([W]_{A_{1}}+e\right),\,[W]_{A_{\infty,1}^{sc}}\right\} .
\end{align*}
If we now combine all the preceding estimates, we complete the proof. 

\subsection{Proof of Theorem \ref{eq:HormanderEndpoint}}
We will follow ideas in \cite[p. 2544]{LPRRR}. By duality for Lorentz-Bochner
spaces
\[
\left\Vert |W^{\frac{1}{r}}(x)T(W^{-\frac{1}{r}}\vec{f})(x)|\right\Vert _{L^{r,\infty}(\mathbb{R}^{d})}=\sup_{\||\vec{g}|\|_{L^{r',1}(\mathbb{R}^{d})}=1}\left|\int_{\mathbb{R}^{d}}\langle W^{\frac{1}{r}}(x)T(W^{-\frac{1}{r}}f)(x),\vec{g}(x)\rangle dx\right|.
\]
Hence it suffices to bound the right-hand side. First note that by
sparse domination
\begin{align*}
\left|\int_{\mathbb{R}^{d}}\langle W^{\frac{1}{r}}(x)T(W^{-\frac{1}{r}}f)(x),\vec{g}(x)\rangle dx\right| & =\left|\int_{\mathbb{R}^{d}}\langle T(W^{-\frac{1}{r}}f)(x),W^{\frac{1}{r}}(x)\vec{g}(x)\rangle dx\right|\\
 & \leq c_{n,d,T}\sum_{Q}\langle\langle W^{-\frac{1}{r}}\vec{f}\rangle\rangle_{r,Q}\langle\langle W^{\frac{1}{r}}\vec{g}\rangle\rangle_{1,Q}|Q|.
\end{align*}
Now we observe that choosing $\alpha=1+\frac{1}{2^{d+11}[W]_{A_{\infty,1}^{sc}}}$ we can argue as follows 
\begin{align*}
 & \langle\langle W^{-\frac{1}{r}}\vec{f}\rangle\rangle_{r,Q}\langle\langle W^{\frac{1}{r}}\vec{g}\rangle\rangle_{1,Q}|Q|\\
 & =\langle\langle\mathcal{W}_{1,Q}^{\frac{1}{r}}W^{-\frac{1}{r}}\vec{f}\rangle\rangle_{r,Q}\langle\langle\mathcal{W}_{1,Q}^{-\frac{1}{r}}W^{\frac{1}{r}}\vec{g}\rangle\rangle_{1,Q}|Q|\\
 & \leq\left(\frac{1}{|Q|}\int_{Q}|\mathcal{W}_{1,Q}^{\frac{1}{r}}W^{-\frac{1}{r}}\vec{f}|^{r}\right)^{\frac{1}{r}}\left(\frac{1}{|Q|}\int_{Q}|\mathcal{W}_{1,Q}^{-\frac{1}{r}}W^{\frac{1}{r}}\vec{g}|\right)|Q|\\
 & \leq\left(\frac{1}{|Q|}\int_{Q}|\mathcal{W}_{1,Q}^{\frac{1}{r}}W^{-\frac{1}{r}}|_{op}^{r}|\vec{f}|^{r}\right)^{\frac{1}{r}}\left(\frac{1}{|Q|}\int_{Q}|\mathcal{W}_{1,Q}^{-\frac{1}{r}}W^{\frac{1}{r}}|_{op}|\vec{g}|\right)|Q|\\
 & \leq c_{n}\left(\frac{1}{|Q|}\int_{Q}|\mathcal{W}_{1,Q}W|_{op}|\vec{f}|^{r}\right)^{\frac{1}{r}}\left(\frac{1}{|Q|}\int_{Q}|\mathcal{W}_{1,Q}^{-1}W|_{op}^{\frac{1}{r}}|\vec{g}|\right)|Q|\\
 & \leq c_{n}[W]_{A_{1}}^{\frac{1}{r}}\left(\frac{1}{|Q|}\int_{Q}|\vec{f}|^{r}\right)^{\frac{1}{r}}\left(\frac{1}{|Q|}\int_{Q}|\mathcal{W}_{1,Q}^{-1}W|_{op}^{\alpha}\right)^{\frac{1}{\alpha r}}\left(\frac{1}{|Q|}\int_{Q}|\vec{g}|^{(\alpha r)'}\right)^{\frac{1}{(\alpha r)'}}|Q|\\
 & \leq c_{n}[W]_{A_{1}}^{\frac{1}{r}}\left(\frac{1}{|Q|}\int_{Q}|\vec{f}|^{r}\right)^{\frac{1}{r}}\left(\frac{1}{|Q|}\int_{Q}|\mathcal{W}_{1,Q}^{-1}W|_{op}^{\alpha}\right)^{\frac{1}{\alpha r}}\left(\frac{1}{|Q|}\int_{Q}|\vec{g}|^{(\alpha r)'}\right)^{\frac{1}{(\alpha r)'}}|Q|\\
 & \leq c_{n}[W]_{A_{1}}^{\frac{1}{r}}\left(\frac{1}{|Q|}\int_{Q}|\vec{f}|^{r}\right)^{\frac{1}{r}}\left(\frac{1}{|Q|}\int_{Q}|\vec{g}|^{(\alpha r)'}\right)^{\frac{1}{(\alpha r)'}}|Q|
\end{align*}
Taking these computations and our sparse domination result into account, 
\begin{align*}
\left|\int_{\mathbb{R}^{d}}\langle W^{\frac{1}{r}}(x)T(W^{-\frac{1}{r}}f)(x),\vec{g}(x)\rangle dx\right| & \lesssim[W]_{A_{1}}^{\frac{1}{r}}\sum_{Q\in\mathcal{S}}\left(\frac{1}{|Q|}\int_{Q}|\vec{f}|^{r}\right)^{\frac{1}{r}}\left(\frac{1}{|Q|}\int_{Q}|\vec{g}|^{r\alpha}\right)^{\frac{1}{(r\alpha)'}}|Q|\\
 & \leq[W]_{A_{1}}^{\frac{1}{r}}\int\sum_{Q}\left(\frac{1}{|Q|}\int_{Q}|\vec{f}|^{r}\right)^{\frac{1}{r}}\chi_{Q}(x)M_{(r\alpha)'}(|\vec{g}|)(x)dx\\
 & \leq[W]_{A_{1}}^{\frac{1}{r}}\left\Vert \sum_{Q}\left(\frac{1}{|Q|}\int_{Q}|\vec{f}|^{r}\right)^{\frac{1}{r}}\chi_{Q}\right\Vert _{L^{r,\infty}}\left\Vert M_{(r\alpha)'}(|\vec{g}|)\right\Vert _{L^{r',1}}\\
 & \lesssim[W]_{A_{1}}^{\frac{1}{r}}\left(\frac{r'}{(r\alpha)'}\right)^{'}\||\vec{f}|\|_{L^{r}}\||\vec{g}|\|_{L^{r',1}}\\
 & \lesssim[W]_{A_{1}}^{\frac{1}{r}}[W]_{A_{\infty,1}^{sc}}\||\vec{f}|\|_{L^{r}}\||\vec{g}|\|_{L^{r',1}}
\end{align*}
where
\[
\left\Vert M_{(r\alpha)'}(|\vec{g}|)\right\Vert _{L^{r',1}}\lesssim\left(\frac{r'}{(r\alpha)'}\right)^{'}\||\vec{g}|\|_{L^{r',1}}
\]
was settled in \cite[p. 2544]{LPRRR} and $\left(\frac{r'}{(r\alpha)'}\right)^{'}\lesssim[W]_{A_{\infty,1}^{sc}}$
by the choice of $\alpha$ and Lemma \ref{lem:param1}. This ends
the proof.

\section{Acknowledgement}
This work will be part of the first author's PhD thesis at Universidad Nacional del Sur. This research was partially supported by Agencia I+D+i PICT 2018-02501 and PICT 2019-00018.


\begin{thebibliography}{10}
\bibitem{Be}Berra, Fabio. From $A_1$ to $A_\infty$: New Mixed Inequalities for Certain Maximal Operators. Available online in Potential Analysis.

\bibitem{B}Bownik, Marcin. Inverse volume inequalities for matrix
weights. Indiana Univ. Math. J. 50 (2001), no. 1, 383\textendash 410.

\bibitem{C}Canto, Javier. Sharp reverse H\"older inequality for $C_p$ weights and applications. Available online in J. Geom. Anal.

\bibitem{CLRT}Canto, Javier; Li, Kangwei; Roncal, Luz and Tapiola, Olli. $C_p$ estimates for rough homogeneous singular integrals and sparse forms. To appear in Ann. Sc. Norm. Super. Pisa Cl. Sci.

\bibitem{CLPRR} Cejas, Mar\'{\i}a Eugenia; Li, Kangwei; Pérez, Carlos; Rivera-R\'{\i}os, Israel Pablo; Vector-valued operators, optimal weighted estimates and the $C_p$ condition. Sci. China Math. 63 (2020), no. 7, 1339\textendash1368.

\bibitem{CUIM}Cruz-Uribe, David; Isralowitz, Joshua and Moen, Kabe.
Integr. Equ. Oper. Theory (2018) 90: 36.

\bibitem{CUIMPRR} Cruz-Uribe, David;  Isralowitz, Joshua; Moen, Kabe; Pott, Sandra; Rivera-R\'{\i}os, Israel P. Weak endpoint bounds for matrix weights (2019). Available online in Revista Matemática Iberoamericana.

\bibitem{CDiPOu} Culiuc, Amalia; Di Plinio, Francesco; Ou, Yumeng Uniform sparse domination of singular integrals via dyadic shifts. Math. Res. Lett. 25 (2018), no. 1, 21–42.

\bibitem{DiPHL}Di Plinio, Francesco; Hytönen, Tuomas P. and Li, Kangwei.
Sparse bounds for maximal rough singular integrals via the Fourier
transform. To appear in Annales de l'institut Fourier.

\bibitem{DKPS} Domelevo, Komla; Kakaroumpas, Spyridon; Petermichl, Stefanie; Soler i Gibert, Od\'{\i}. Boundedness of Journé operators with matrix weights. math.CA arXiv:2102.03395

\bibitem{FR}Frazier, Michael; Roudenko, Svetlana. Matrix-weighted
Besov spaces and conditions of $A_{p}$ type for $0<p\leq1$. Indiana
Univ. Math. J. 53 (2004), no. 5, 1225\textendash 1254.

\bibitem{F}Fujii, Nobuhiko. Weighted bounded mean oscillation and
singular integrals. Math. Japon., 22 (5):529\textendash 534, 1977/78.

\bibitem{G}Goldberg, Michael. Matrix $A_{p}$ weights via maximal
functions., Pacific J. Math. 211, no. 2 (2003), 201 \textendash 220.

\bibitem{H}Hytönen, Tuomas. The sharp weighted bound for general
Calderón-Zygmund operators. Ann. of Math. (2) 175 (2012), no. 3, 1473\textendash 1506.


\bibitem{HPAinfty}Hytönen, Tuomas; Pérez, Carlos. Sharp weighted
bounds involving $A_{\infty}$. Anal. PDE 6 (2013), no. 4, 777\textendash 818.

\bibitem{HPR}Hytönen, Tuomas; Pérez, Carlos; Rela, Ezequiel Sharp
reverse Hölder property for $A_{\infty}$ weights on spaces of homogeneous
type. J. Funct. Anal. 263 (2012), no. 12, 3883\textendash 3899.

\bibitem{HPV}Hytönen, Tuomas; Petermichl, Stefanie and Volberg, Alexander.
The sharp square function estimate with matrix weight. Discrete Analysis 2019:2, 8 pp.

\bibitem{IFRR}Iba\~{n}ez-Firnkorn, Gonzalo H.; Rivera-R\'{\i}os, Israel P. Sparse and weighted estimates for generalized Hörmander operators and commutators. Monatsh. Math. 191 (2020), no. 1, 125\textendash173.

\bibitem{IKP}Isralowitz, Joshua; Kwon, Hyun-Kyoung; Pott, Sandra.
Matrix weighted norm inequalities for commutators and paraproducts
with matrix symbols. J. Lond. Math. Soc. (2) 96 (2017), no. 1, 243\textendash 270.

\bibitem{IPT}Isralowitz, Joshua; Pott, Sandra; Treil, Sergei. Commutators in the two scalar and matrix weighted setting. math.CA arXiv:2001.11182 .

\bibitem{IPRR} Isralowitz, Joshua; Pott, Sandra; Rivera-R\'{\i}os, Israel P.; Sharp $A_1$ Weighted Estimates for Vector-Valued Operators. J. Geom. Anal. 31 (2021), no. 3, 3085\textendash3116.

\bibitem{LeCF1} Lerner, Andrei K. Some remarks on the Fefferman-Stein inequality. J. Anal. Math. 112 (2010), 329\textendash349. 

\bibitem{L} Lerner, Andrei K. On an estimate of Calderón-Zygmund operators by dyadic positive operators. J. Anal. Math. 121 (2013), 141\textendash161.

\bibitem{Le} Lerner, Andrei K. A weak type estimate for rough singular integrals. Rev. Mat. Iberoam. 35 (2019), no. 5, 1583\textendash1602. 

\bibitem{LeCF2} Lerner, Andrei K. A characterization of the weighted weak type Coifman-Fefferman and Fefferman-Stein inequalities. Math. Ann. 378 (2020), no. 1-2, 425\textendash446.

\bibitem{LN} Lerner, Andrei K.; Nazarov, Fedor Intuitive dyadic calculus: the basics. Expo. Math. 37 (2019), no. 3, 225\textendash265.

\bibitem{LNO}Lerner, Andrei K.; Nazarov, Fedor; Ombrosi, Sheldy On the sharp upper bound related to the weak Muckenhoupt–Wheeden conjecture. Anal. PDE 13 (2020), no. 6, 1939\textendash1954.


\bibitem{LOP}Lerner, Andrei K.; Ombrosi, Sheldy; Pérez, Carlos Sharp
$A_{1}$ bounds for Calderón-Zygmund operators and the relationship
with a problem of Muckenhoupt and Wheeden. Int. Math. Res. Not. IMRN
2008, no. 6, Art. ID rnm161, 11 pp.

\bibitem{LOP1}Lerner, Andrei K.; Ombrosi, Sheldy; Pérez, Carlos $A_{1}$
bounds for Calderón-Zygmund operators related to a problem of Muckenhoupt
and Wheeden. Math. Res. Lett. 16 (2009), no. 1, 149\textendash 156.

\bibitem{Li}Li, Kangwei Sparse domination theorem for multilinear singular integral operators with $L^r$-H\"ormander condition. Michigan Math. J. 67 (2018), no. 2, 253\textendash265.

\bibitem{LiOP}Li, Kangwei; Ombrosi, Sheldy; Pérez, Carlos Proof of an extension of E. Sawyer's conjecture about weighted mixed weak-type estimates. Math. Ann. 374 (2019), no. 1-2, 907–929.


\bibitem{LPRRR}Li, Kangwei; Pérez, Carlos; Rivera-Ríos, Israel P.; Roncal, Luz Weighted norm inequalities for rough singular integral operators. J. Geom. Anal. 29 (2019), no. 3, 2526\textendash2564.

\bibitem{MPTG}Martell, José María; Pérez, Carlos; Trujillo-González, Rodrigo Lack of natural weighted estimates for some singular integral operators. Trans. Amer. Math. Soc. 357 (2005), no. 1, 385\textendash396.

\bibitem{M}B. Muckenhoupt, Norm inequalities relating the Hilbert ransform to the Hardy–Littlewood maximal function, in Functional Analysis and Approximation (Oberwolfach, 1980), Birkhäuser, Basel–Boston, Mass., 1981, pp. 219\textendash231.

\bibitem{MW} Muckenhoupt, Benjamin; Wheeden, Richard L. Some weighted weak-type inequalities for the Hardy-Littlewood maximal function and the Hilbert transform. Indiana Univ. Math. J. 26 (1977), no. 5, 801–816.

\bibitem{NPTV}Nazarov, Fedor; Petermichl, Stefanie; Treil, Sergei;
Volberg, Alexander Convex body domination and weighted estimates with
matrix weights. Adv. Math. 318 (2017), 279\textendash 306.



\bibitem{NT} Nazarov, Fedor and Treil, Sergei. The hunt for a Bellman
function: applications to estimates for singular integral operators
and to other classical problems of harmonic analysis, Algebra i Analiz
8 (1996) 32 \textendash 162.


\bibitem{OCPR}Ortiz-Caraballo, Carmen; P\'erez, Carlos; Rela, Ezequiel Improving bounds for singular operators via sharp reverse Hölder inequality for $A_\infty$. Advances in harmonic analysis and operator theory, 303\textendash321, Oper. Theory Adv. Appl., 229, Birkhäuser/Springer Basel AG, Basel, 2013. 

\bibitem{R}Roudenko, Svetlana. Matrix-weighted Besov spaces. Trans.
Amer. Math. Soc. 355 (2003), no. 1, 273\textendash 314.

\bibitem{S} Sawyer, Eric T. Norm inequalities relating singular integrals and the maximal function. Studia Math. 75 (1983), no. 3, 253\textendash263.

\bibitem{S2} Sawyer, Eric T.
A weighted weak type inequality for the maximal function.
Proc. Amer. Math. Soc. 93 (1985), no. 4, 610\textendash614.

\bibitem{TV}Treil, Sergei and Volberg, Alexander. Wavelets and the
Angle between Past and Future, J. Funct. Anal. 143 (1997), no. 2,
269\textendash308



\bibitem{V}Volberg, Alexander. Matrix $A_{p}$ weights via $S$-functions,
J. Amer. Math. Soc. 10 (1997), no. 2, 445\textendash 466.

\bibitem{W}Wilson, J. Michael. Weighted inequalities for the dyadic
square function without dyadic $A_{\infty}$. Duke Math. J. 55 (1987),
no. 1, 19\textendash 50.

\bibitem{Y}Yabuta, Kozo. 
Sharp maximal function and $C_p$ condition.
Arch. Math. (Basel) 55 (1990), no. 2, 151\textendash155.

\end{thebibliography}
\end{document}